\documentclass[a4paper, 11pt, final]{article}

\usepackage{amssymb,amsmath,amsthm,mathrsfs}
\usepackage{enumitem}  
\usepackage{a4wide}
\usepackage{hyperref}
\usepackage{array}

\numberwithin{equation}{section}
\allowdisplaybreaks[4]

\newtheorem{proposition}{Proposition}[section]
\newtheorem{theorem}[proposition]{Theorem}
\newtheorem{lemma}[proposition]{Lemma}
\newtheorem{definition}[proposition]{Definition}

\renewenvironment{proof}{\smallskip\noindent\emph{\textbf{Proof.}}%
  \hspace{1pt}}{\hspace{-5pt}{\nobreak\quad\nobreak\hfill\nobreak%
    $\square$\vspace{2pt}\par}\smallskip\goodbreak}

\newenvironment{proofof}[1]{\smallskip\noindent{\textbf{Proof~of~#1.}}%
  \hspace{1pt}}{\hspace{-5pt}{\nobreak\quad\nobreak\hfill\nobreak%
    $\square$\vspace{2pt}\par}\smallskip\goodbreak}

\newcommand{\C}[1]{\mathbf{C}^{#1}}
\newcommand{\Cc}[1]{\mathbf{C}_c^{#1}}

\newcommand{\BV}{\mathbf{BV}}

\renewcommand{\L}[1]{{\mathbf{L}^#1}}
\newcommand{\Lloc}[1]{{\mathbf{L}_{\mathbf{loc}}^{#1}}}

\newcommand{\modulo}[1]{{\left|#1\right|}}
\newcommand{\norma}[1]{{\left\|#1\right\|}}
\newcommand{\caratt}[1]{{\chi_{\strut#1}}}
\newcommand{\reali}{{\mathbb{R}}}

\renewcommand{\epsilon}{\varepsilon}
\renewcommand{\phi}{\varphi}
\renewcommand{\theta}{\vartheta}
\renewcommand{\O}{\mathcal{O}(1)}
\newcommand{\tv}{\mathinner{\rm TV}}
\newcommand{\spt}{\mathop{\rm spt}}

\renewcommand{\d}[1]{\mathinner{\mathrm{d}{#1}}}

\makeatletter
\let\@fnsymbol\@arabic
\makeatother

\title{Non Conservative Products in Fluid Dynamics}

\author{Rinaldo M. Colombo\footnotemark[1] \and Graziano Guerra\footnotemark[2] \and Yannick Holle\footnotemark[3]}

\begin{document}

\maketitle

\footnotetext[1]{INdAM Unit \& Department of Information Engineering,
  University of Brescia, via Branze, 38, 25123 Brescia,
  Italy. e--mail: \texttt{rinaldo.colombo@unibs.it}. ORCID:
  \texttt{0000-0003-0459-585X}.} \footnotetext[2]{Universit\`a degli
  Studi di Milano Bicocca, Dipartimento di Matematica e Applicazioni,
  via R.~Cozzi, 55, 20126 Milano, Italy. e--mail:
  \texttt{graziano.guerra@unimib.it}. ORCID:
  \texttt{0000-0003-2615-2750}.} \footnotetext[3]{RWTH Aachen
  University, Institut f\"ur Mathematik, Templergraben 55, 52062
  Aachen, Germany. e--mail: \texttt{holle@eddy.rwth-aachen.de}. ORCID:
  \texttt{0000-0002-7841-8525}.\vspace{0.05cm}}

\begin{abstract}
  \noindent Fluid flow in pipes with discontinuous cross section or
  with kinks is described through balance laws with a non conservative
  product in the source. At jump discontinuities in the pipes'
  geometry, the physics of the problem suggests how to single out a
  solution. On this basis, we present a definition of solution for a
  general $\BV$ geometry and prove an existence result, consistent
  with a limiting procedure from piecewise constant geometries. In the
  case of a smoothly curved pipe we thus justify the appearance of the
  curvature in the source term of the linear momentum equation.

  These results are obtained as consequences of a general existence
  result devoted to abstract balance laws with non conservative source
  terms.

  \medskip


\end{abstract}

\section{Introduction}
\label{sec:Intro}

Conservation laws in one space dimension, i.e., systems of partial
differential equations in conservative form of the type
\begin{equation}
  \label{eq:GeneralSystem}
  \partial_t u + \partial_x f(u) = 0
  \qquad t\geq 0,\; x \in \reali \,,
\end{equation}
allow to describe, for instance, the movement of a fluid along a
rectilinear pipe with constant section. Assume that at a point
$\bar x$ the pipe's direction or its section changes. Then,
equation~\eqref{eq:GeneralSystem} can be used, separately, where
$x < \bar x$ and where $x > \bar x$. At the point $\bar x$, on the
basis of physical considerations, a further condition is necessary to
prescribe the possible defect in the conservation of the various
variables. Typically, such a condition is written as
\begin{equation}
  \label{eq:IntroSingleCoupling}
  \Psi\left(z^+,u(t,\bar x+),z^-,u(t,\bar x-) \right) = 0
  \quad \mbox{for a.e. }t>0,
\end{equation}
where $z^+$ and $z^-$, identify the physical parameters that change
across $\bar x$. Alternatively, \eqref{eq:IntroSingleCoupling} can be
rewritten making the defect in the conservation of the $u$ variable
explicit, that is
\begin{equation}
  \label{eq:DefPsiCoupling}
  f\left(u(t,\bar x+)\right) - f\left(u(t,\bar x-)\right)
  =
  \Xi\big(z^+,z^-,u(t, \bar x-)\big)
  \quad \mbox{for a.e. }t>0 \,.
\end{equation}
It is then natural to tackle the resulting Riemann Problem, that is,
the Cauchy Problem consisting
of~\eqref{eq:GeneralSystem}--\eqref{eq:DefPsiCoupling} with an initial
datum attaining two values, one for $x<0$ and another one for $x>0$,
as was accomplished, for instance,
in~\cite[\S~2]{AmadoriGosseGuerra2002}
or~\cite[\S~2]{ColomboGaravello3}. The finite propagation speed,
intrinsic to~\eqref{eq:GeneralSystem}, allows then to extend the whole
construction to any finite number of points
$\bar x_0, \bar x_1, \ldots, \bar x_k$, essentially solving the Cauchy
Problem for the balance law
\begin{equation}
  \label{eq:68}
  \left\{
    \begin{array}{l}
      \partial_t u+\partial_x f(u)
      =
      \sum\limits_{i=1}^{k-1}
      \Xi \left(\zeta_k (\bar x_i+), \zeta_k (\bar x_i-), u (t, \bar x_i-)\right) \;
      \delta_{\bar x_i}
      \\
      u (0,x) = u_o (x),
    \end{array}
  \right.
\end{equation}
where $\delta_{\bar x_i}$ denotes the Dirac measure at $\bar x_i$ and
$\zeta_k$ is the piecewise constant function attaining the $k+1$
constant values $z_0, z_1, \ldots, z_k$ on the intervals
$\mathopen]-\infty, \bar x_1 \mathclose[$,
$\mathopen]\bar x_1, \bar x_2 \mathclose[$, $\ldots$,
$\mathopen]\bar x_k, +\infty \mathclose[$.

This paper provides a detailed description of the rigorous limit
$k\to +\infty$ of~\eqref{eq:68}, covering the extension
of~\eqref{eq:68} to the case of a general $\BV$ function $\zeta$.

In the general setting established below, not limited to fluid
dynamics, solutions to~\eqref{eq:68} with initial datum $u_o$ are
shown to converge as $k \to +\infty$ to solutions to
\begin{equation}
  \label{eq:21}
  \left\{
    \begin{array}{l@{}}
      \partial_t u
      +
      \partial_x f(u)
      =
      \sum\limits_{x \in \mathcal I}
      \Xi\left(\zeta (x+), \zeta (x-), u (t,x-)\right)
      \delta_{x}
      +
      D^+_{v (x)} \Xi\left(\zeta (x), \zeta (x), u (t,x)\right)
      \norma{\mu}
      \\
      u (0,x)
      =
      u_o (x) \,.
    \end{array}
  \right.
\end{equation}
The terms in the non conservative source above are defined as
follows. Since $\zeta \in \BV (\reali; \reali^p)$, the right and left
limits $\zeta (x+)$ and $\zeta (x-)$ are well defined and the
distributional derivative $D\zeta$ can be split in a discrete part and
a continuous one, which may contain a Cantor part:
\begin{equation}
  \label{eq:20}
  D\zeta
  =
  \sum_{x \in \mathcal{I}} \left(\zeta (x+) - \zeta (x-)\right) \, \delta_x
  +
  v \, \norma{\mu} \,,
\end{equation}
where the function $v$ is Borel measurable with norm $1$, $\mu$ is the
non atomic part of $D\zeta$ and $\mathcal{I}$ is the set of jump
points in $\zeta$.  In~\eqref{eq:21} we also used the (one sided)
directional derivative
\begin{equation}
  \label{eq:26}
  D^+_{v} \Xi (z, z, u)
  =
  \lim_{t \to 0+}
  \dfrac{
    \Xi(z + t \, v, z, u)
    -
    \Xi(z, z, u)}{t} \,.
\end{equation}
Indeed, one of our motivating examples, namely the case of a curved
pipe, leads to a function $\Xi$ that admits directional derivatives
but is not differentiable.

On the other hand, note that as soon as $\Xi$ is differentiable with
respect to its first argument, the right hand side in~\eqref{eq:21}
can be slightly simplified, since
\begin{equation}
  \label{eq:55}
  D^+_{v (x)} \Xi (a,a,u)\, \norma{\mu}
  =
  D_1 \Xi (a,a,u)\, v (x) \, \norma{\mu}
  =
  D_1 \Xi (a,a,u)\, \mu \,.
\end{equation}
Moreover, in the case $\Xi (z^{+}, z^{-},u) = G(z^{+}) - G (z^{-})$
for a suitable $G \in \C2 (\reali^p; \reali^n)$, the right hand side
above takes a simpler form. Indeed,
by~\cite[Theorem~3.96]{AmbrosioFuscoPallara}, \eqref{eq:21} reduces to
the conservative problem
\begin{equation}
  \label{eq:57}
  \partial_t u
  +
  \partial_x f(u)
  =
  \partial_x (G\circ \zeta) \,.
\end{equation}

Below, our first task is to provide a definition of solution
to~\eqref{eq:21} in its general setting. Indeed, the latter term in
the right hand side of~\eqref{eq:21} contains a \emph{non conservative
  product} between a possibly discontinuous function and a measure. As
is well known since the pioneering work~\cite{DalMasoMurat1995}, such
a product intrinsically contains a lack of determinacy. Here, this
freedom of choice is used to ensure the convergence of~\eqref{eq:68}
to~\eqref{eq:21}.

Once the issue of the very meaning of solution is settled, we proceed
towards proving the existence of solutions to~\eqref{eq:21}. This is
achieved sequentially combining wave front
tracking~\cite[\S~7.1]{Bressan2000}, a nowadays classical technique
that approximates solutions to conservation laws, with the
approximation of the equation, in particular of the map $\zeta$. A key
role is played by a very careful choice of these approximations. As a
byproduct, we characterize the solutions to~\eqref{eq:21} as limits of
(suitable subsequences of) solutions to~\eqref{eq:68}.

\smallskip

Remark that the above general procedure, when applied to the case of a
curved pipe with constant section, amounts to justify the role of the
pipe's curvature on the fluid flow inside the pipe. Indeed, if $x$ is
the abscissa along the pipe and $\Gamma = \Gamma (x)$ describes the
pipe's shape, then the pipe's local direction that enters the equation
for fluid flow is $\zeta (x) = \Gamma' (x)$. Problem~\eqref{eq:68}
then corresponds to a piecewise linear pipe and (the second component
of)~\eqref{eq:DefPsiCoupling} describes the change in the fluid linear
momentum at a kink sited at $\bar x$. Assuming that the lack in the
conservation of linear momentum depends on the angle in the pipe at
$\bar x$, i.e., $\Xi (z^+, z^-, u) = K (\norma{z^+-z^-}, u)$ as
in~\cite{ColomboHolden2016, HoldenRisebro1999}, automatically implies
in the smooth pipe limit, by Theorem~\ref{thm:main}, that the
variation in the fluid momentum depends on the pipe's curvature
$\Gamma''$, see~\S~\ref{sub:ApplCurvedPipeIsentropic} for more
details.

The current literature offers a variety of different conditions
quantifying the lack in the conservation of linear momentum at a
junction where the pipe's section changes, see for
instance~\cite{BandaHertyKlar1, BandaHertyKlar2, ColomboGaravello2007,
  ColomboGaravello3, ColomboGaravello2020, ColomboMarcellini2010Euler,
  ColomboMarcellini2010, GuerraMarcelliniSchleper2009}. As a
consequence of Theorem~\ref{thm:main}, we can select those conditions
that are consistent with the equations for a pipe with smoothly
varying section, both in the isentropic and in the full $3\times3$
cases, see \S~\ref{sec:isentr-gas-pipe}
and~\S~\ref{sec:full-gas-dynamics} below.

While motivated by the above fluid dynamics problems, the present
construction also suggests a criterion to select solutions to general
balance laws with a non conservative product as a source term, see
Definition~\ref{def:NonCons}. These solutions, whose existence follows
from Theorem~\ref{thm:main}, are characterized as limits of solutions
to the piecewise constant case~\eqref{eq:68}.

\smallskip

The next section is devoted to the main results: the definition of
solution and to the existence theorem. Section~\ref{sec:Appl} presents
applications to fluid dynamics and to general balance laws with non
conservative product in the source. All technical proofs are deferred
to Section~\ref{sub:GerneralSingle}.

\section{Assumptions and Main Result}
\label{sec:General}

Throughout, $\modulo{x}$ is the absolute value of the real number $x$
while, as usual, $\norma{v}$ is the Euclidean norm of the vector $v$
and $\norma{\mu}$ is the total variation of a measure $\mu$.  The open
ball in $\reali^n$ centered at $u$ with radius $\delta$ is denoted by
$B(u;\delta)$, its closure is $\overline{B (u; \delta)}$. We also use
the following standard notation for right/left limits and for
differences at a point:
\begin{displaymath}
  F(x-)=\lim_{\xi\to x^{-}}F(\xi)\, ,\quad
  F(x+)=\lim_{\xi\to x^{+}}F(\xi)
  \quad \mbox{ and } \quad
  \Delta F(x) = F(x+)-F(x-) \,.
\end{displaymath}

The problem we tackle is defined by the flow $f$ and by the functions
$\Xi$ and $\zeta$. Here we detail the key assumptions.

\begin{enumerate}[label={\bf{(f.\arabic*)}}]
\item \label{it:f1} $f\in \C2(\Omega; \reali^n)$, $\Omega$ being an
  open subset of $\reali^n$;
\item \label{it:f2} the system~\eqref{eq:GeneralSystem} is strictly
  hyperbolic;
\item \label{it:f3} each characteristic field is either genuinely
  nonlinear or linearly degenerate.
\end{enumerate}
\noindent In the latter assumption we refer to the usual definitions
by Lax~\cite{Lax1957}, see also~\cite[\S~7.5]{Dafermos2000}.

By~\ref{it:f1} and~\ref{it:f2} we know that, possibly restricting
$\Omega$, the eigenvalues $\lambda_1(u),\dots,\lambda_n(u)$ of $Df(u)$
can be numbered so that, for all $u \in \Omega$,
\begin{displaymath}
  \lambda_1(u) < \lambda_2(u) < \dots < \lambda_n(u) \,.
\end{displaymath}
We choose $i_o \in \{1,\dots,n-1\}$ and define the $i_o$-th
non-characteristic set
\begin{equation}
  \label{eq:23}
  A_{i_o} = \{u\in\Omega \;| \; \lambda_{i_o}(u)<0<\lambda_{i_o+1}(u)\} \,,
\end{equation}
both the cases of characteristic speeds being either all positive or
all negative being simpler.

On the function $\Xi$ in~\eqref{eq:DefPsiCoupling}, used to rewrite
the coupling condition induced by $\Psi$, we require:

\begin{enumerate}[label={\bf{($\mathbf{\Xi}$.\arabic*)}}]
\item \label{eq:Xi1}
  $\Xi : \mathcal{Z}\times\mathcal{Z}\to \C1\left(\Omega;
    \reali^n\right)$, is a Lipschitz continuous map;
\item \label{eq:Xi2}
  $\sup_{z^+,z^-\in\mathcal Z}\norma{\Xi(z^+, z^-,
    \cdot)}_{\C{2}(\Omega; \reali)}< +\infty$;
\item \label{eq:Xi3} $\Xi(z, z, u) = 0$ for every $z \in \mathcal{Z}$
  and $u \in \Omega$;
\item \label{eq:Xi4} There exists a non decreasing
  $\sigma \colon \mathopen[0, \bar t\mathclose[ \to \reali$ such that
  for all
  $(z,v,u) \in \mathcal{Z} \times \overline{B(0;1)} \times \Omega$
  \begin{displaymath}
    \norma{\Xi (z + t \, v, z, u) - D^+_v\Xi (z,z,u) \, t}
    \leq
    \sigma (t) \, t
  \end{displaymath}
  and moreover the map $(z,v,u) \to D^+_v \Xi (z,z,u)$ is Lipschitz
  continuous.
\end{enumerate}
\noindent In the latter condition, recall the definition~\eqref{eq:26}
of the Dini derivative. Our requiring this low regularity, i.e.~the
mere existence of the Dini derivative rather than differentiability,
is motivated by the example of a pipe with angles, where $\Xi$ depends
on $\norma{z^+ - z^-}$, see \S~\ref{sub:ApplCurvedPipeIsentropic}.

Problem~\eqref{eq:21} requires the introduction of a further function,
say $\zeta \colon \reali \to \reali^p$ describing, for instance,
geometrical aspects of the pipeline. We require that
$\zeta \in \BV (\reali; \mathcal{Z})$. Throughout, the map $\zeta$ is
assumed to be left continuous and the set of jump discontinuities in
$\zeta$ is denoted by $\mathcal{I}$, with
$\mathcal{I} \subset \reali$.

We now precisely state what we mean by \emph{solution}
to~\eqref{eq:21}.

\begin{definition}
  \label{def:sol}
  Let $u_o \in \Lloc1 (\reali; \reali^n)$.  A map
  $u \in \C0 (\mathopen[0, +\infty\mathclose[; \Lloc1 (\reali;
  \reali^n))$ with $u (t) \in \BV (\reali; \reali^n)$ and left
  continuous for all $t \in \reali_+$, is a \emph{solution}
  to~\eqref{eq:21} if for all test functions
  $\phi \in \Cc1 (\mathopen]0, +\infty\mathclose[ \times \reali;
  \reali)$,
  \begin{eqnarray}
    \nonumber
    &
    & -\int_0^{+\infty} \int_{\reali}
      \left(
      u (t,x) \, \partial_t \phi (t,x)
      +
      f\left(u (t,x)\right) \, \partial_x \phi (t,x)
      \right)
      \d{x} \d{t}
    \\
    \label{eq:48}
    & =
    & \sum_{x \in \mathcal{I}}
      \int_0^{+\infty}
      \Xi\left(\zeta (x+), \zeta (x), u (t, x)\right)
      \phi (t,x) \d{t}
    \\
    \nonumber
    &
    & +
      \int_0^{+\infty} \int_{\reali}
      D_{v (x)}^+ \Xi\left(\zeta (x), \zeta (x), u (t,x)\right)
      \phi (t,x) \, \d{\norma{\mu}}(x)\, \d{t}
  \end{eqnarray}
  where $\mathcal{I}$ is the set of jump points of $\zeta$ and $v$,
  $\mu$ are as in~\eqref{eq:20}, and moreover $u (0) = u_o$.
\end{definition}

The main result of this paper is the following.

\begin{theorem}
  \label{thm:main}
  Let $\Omega \subseteq \reali^n$ be open, $f$
  satisfy~\ref{it:f1}--\ref{it:f3}, $\Xi$
  satisfy~\ref{eq:Xi1}--\ref{eq:Xi4} and
  $\zeta \in \BV (\reali; \mathcal{Z})$. Fix $\bar u \in A_{i_o}$,
  $\bar z \in \mathcal{Z}$ and an initial datum $u_o$ in
  $\Lloc1 (\reali;A_{i_o})$, with $A_{i_o}$ as defined
  in~\eqref{eq:23}. Then, there exists a positive $\delta$ such that
  if
  \begin{equation}
    \label{eq:22}
    u_o (\reali) \subseteq B (\bar u; \delta) \,,\quad
    \tv (u_o) < \delta
    \quad \mbox{ and } \quad
    \zeta (\reali) \subseteq B (\bar z; \delta) \,,\quad
    \tv (\zeta) < \delta
  \end{equation}
  the Cauchy Problem for~\eqref{eq:21} with initial datum $u_o$ admits
  a solution $u_*$ in the sense of Definition~\ref{def:sol}. Moreover,
  there exists a sequence of piecewise constant approximations
  $\zeta^h$ of $\zeta$, with $\tv (\zeta^h) < \delta$, such that the
  corresponding solutions $u^h$ converge to $u_*$ pointwise in time
  and in $\Lloc1$ in space. In particular, at each discontinuity point
  $y$ of $\zeta^h$, $u^h$ satisfies the junction condition
  \begin{displaymath}
    f(u^h(t,y+)) - f(u^h(t,y-))
    = \Xi(\zeta^h (y+),\zeta^h (y-),u^h(t,y-)) \,.
  \end{displaymath}
\end{theorem}

\section{Applications}
\label{sec:Appl}

\subsection{Isentropic Gas in a Curved Pipe}
\label{sub:ApplCurvedPipeIsentropic}

The well known system of one dimensional isentropic gas dynamics
within a pipe with constant section in Eulerian
coordinates~\cite[Formula~(7.1.12)]{Dafermos2000} is
\begin{equation}
  \label{eq:IsentropicGas}
  \left\{
    \begin{array}{rcl}
      \partial_t \rho + \partial_x q
      & =
      & 0
      \\
      \partial_t q + \partial_x P (\rho,q)
      & =
      & 0
    \end{array}
  \right.
  \quad \mbox{where }
  P(\rho,q) = \dfrac{q^2}{\rho} + p(\rho) \,,
  \quad \mbox{for a.e. }t\ge 0, \; x\in\reali \,.
\end{equation}
Here, $x$ is the abscissa along the pipe,
$\rho \in \mathopen]0, +\infty\mathclose[$ denotes the gas density,
$q\in\reali$ the momentum density, $p = p(\rho)$ the pressure and
$P = P (\rho,q)$ the momentum flux. The pressure law $p$ satisfies
\begin{enumerate}[label={\bf{(p)}}]
\item\label{it:p}
  $p \in \C2(\left]0, +\infty\right[, \left]0, +\infty\right[)$,
  $p'(\rho) \ge 0$ and $p''(\rho) \ge 0$ for all $\rho>0$.
\end{enumerate}
Under this assumption, system~\eqref{eq:IsentropicGas} is strictly
hyperbolic, except at the vacuum $\rho = 0$.

We aim to estabilish the existence of solutions
to~\eqref{eq:IsentropicGas} in a curved pipeline with constant section
lying in a horizontal plane. Parametrize the pipe's support by means
of the arc length $\Gamma \colon \reali\to \reali^2$, so that
$\norma{\Gamma'(x)} = 1$ for a.e.~$x \in \reali$. We assume that
$\zeta = \Gamma'$ is in $\BV (\reali; \reali^2)$.

As a first step, consider the case of~\eqref{eq:IsentropicGas} at a
kink sited at $\bar x$, so that $\Gamma$ is the glueing of two half
lines. Therefore, to solve~\eqref{eq:IsentropicGas}, we adopt the
usual weak entropy solutions to~\eqref{eq:IsentropicGas} along the
straight parts of $\Gamma$ and match at the kink $\bar x$ a coupling
condition of the type
\begin{equation}
  \label{eq:45}
  \left\{
    \begin{array}{rcl}
      q(t,\bar x+)-q(t,\bar x-)
      & =
      & 0
      \\
      P(\rho,q)(t,\bar x+)-P(\rho,q)(t,\bar x-)
      & =
      & \Xi_2 (\Gamma' (\bar x+), \Gamma' (\bar x-), (\rho,q) (t, \bar x-))
    \end{array}
  \right.
\end{equation}
We set $\Xi_1 \equiv 0$ as it is necessary to comply with mass
conservation. Physical considerations suggest that the defect in the
conservation of linear momentum is a function, say $K$, of the norm of
the difference in the orientations of the pipes on the sides of the
kink:
\begin{equation}
  \label{eq:47}
  \Xi \left(z^+, z^-, (\rho,q)\right)
  =
  \left[
    \begin{array}{c}
      0
      \\
      K \left(\norma{z^+- z^-}, (\rho,q)\right)
    \end{array}
  \right] \,.
\end{equation}
This holds true in various instances of $K$ considered in the
literature. For instance, \cite{HoldenRisebro1999} first introduced
the condition
\begin{equation}
  \label{eq:46}
  K \left(\norma{z^+- z^-}, (\rho,q)\right)
  =
  - \alpha \, \norma{z^+- z^-} \, q
\end{equation}
for a suitable $\alpha > 0$, motivated by
\begin{displaymath}
  \norma{z^+ - z^-}
  =
  \sqrt{2(1-\cos \bar\theta)}
  =
  2 \, \modulo{\sin (\bar\theta/2)} \,,
\end{displaymath}
$\bar\theta$ being the angle between the two sides of the kink. It is
immediate to see that~\ref{eq:Xi1}--\ref{eq:Xi3} all
hold. Concerning~\ref{eq:Xi4}, we have
\begin{displaymath}
  D^+_v \Xi \left(z, z, (\rho,q)\right)
  =
  \left[
    \begin{array}{c}
      0
      \\
      - \alpha \, \norma{v} \, q
    \end{array}
  \right]
  \quad \mbox{ with } \quad \sigma \equiv 0\,.
\end{displaymath}
We stress that $\Xi_2$ is \emph{not} of class $\C1$.

\begin{theorem}
  \label{thm:ConvergenceIsenCurvedPipe}
  Let $p$ satisfy~\ref{it:p} and $(\bar\rho,\bar q)$ be a subsonic
  state. Let $\Gamma$ be piecewise $\C2 (\reali; \reali^2)$, such that
  $\Gamma' \in \BV (\reali; \reali^2)$ and $\norma{\Gamma' (x)} = 1$
  for all $x \in \reali$. Let
  $K \in \C2 ([0, r] \times \Omega; \reali)$ for a positive $r$, with
  $K \left(0,(\rho,q)\right) \equiv 0$. Call $\mathcal{I}$ the set of
  kink points of $\Gamma$. Then, there exists a positive $\delta$ such
  that for all initial data $(\rho_o, q_o)$ with
  \begin{displaymath}
    \norma{(\rho_o,q_o) - (\bar\rho, \bar q)}_{\L\infty (\reali; \reali^2)}
    < \delta
    \,,\quad
    \tv (\rho_o, q_o) < \delta
    \,,\quad
    \tv (\Gamma') < \delta
  \end{displaymath}
  the problem
  \begin{equation}
    \label{eq:49}
    \left\{
      \begin{array}{@{\,}rcl@{}}
        \partial_t\rho + \partial_x q
        & =
        & 0
        \\
        \partial_t q + \partial_x P (\rho,q)
        & =
        & \displaystyle
          - \sum_{y\in\mathcal I}
          K\left(\norma{\Gamma' (y+)-\Gamma' (y-)},(\rho,q)(t,y-)\right)
          \, \delta_{y}
        \\
        &
        & - \norma{\Gamma'' (x)} \, \partial_1 K(0,q)
        \\
        (\rho,q) (0,x)
        & =
        & (\rho_o, q_o) (x)
      \end{array}
    \right.
  \end{equation}
  admits a solution $(\rho_*,q_*)$ in the sense of
  Definition~\ref{def:sol}.  Moreover, there exists a sequence of
  piecewise linear approximations $\Gamma^h$ of $\Gamma$, with
  $\tv ((\Gamma^h)') < \delta$, such that the corresponding solutions
  $(\rho^h, q^h)$ converge to $(\rho_*, q_*)$ pointwise in time and in
  $\Lloc1$ in space. In particular, at each discontinuity point
  $\bar x$ of $(\Gamma^h)'$, $(\rho^h, q^h)$ satisfies
  condition~\eqref{eq:45}.
\end{theorem}

\noindent The proof is deferred to~\S~\ref{sec:Applications}.

Remark that the second derivative $\Gamma''$ appearing in the right
hand side above confirms the relevance of the pipe's
curvature. Nevertheless, Theorem~\ref{thm:main} applies also to less
regular functions $\Gamma$, but the above simpler formulation then
needs to be replaced by the formulation used in
Definition~\ref{def:sol}.

\subsection{Isentropic Gas in a Pipe with Varying Section}
\label{sec:isentr-gas-pipe}

The isentropic flow of a fluid in a pipe with smoothly varying section
$a = a (x)$ is described by
\begin{equation}
  \label{eq:51}
  \left\{
    \begin{array}{@{\,}rcl}
      \partial_t \rho
      + \partial_x q
      & =
      & - \dfrac{a'}{a} \, q
      \\[6pt]
      \partial_t q
      + \partial_x P (\rho,q)
      & =
      & - \dfrac{a'}{a} \, \dfrac{q^2}{\rho}
    \end{array}
  \right.
  \; \mbox{where }
  P(\rho,q) = \dfrac{q^2}{\rho} + p(\rho) \,,
  \quad \mbox{for a.e. }t\ge 0, \; x\in\reali \,,
\end{equation}
see~\cite{ColomboMarcellini2010, GuerraMarcelliniSchleper2009,
  TPLiu1979}. The case of $a$ piecewise constant, i.e., the section of
the pipe changes from $a^-$ to $a^+$ at a junction sited at $\bar x$,
is covered in the literature supplementing the
$p$--system~\eqref{eq:IsentropicGas} with a junction condition of the
form
\begin{equation}
  \label{eq:52}
  \left\{
    \begin{array}{l}
      a^+ q (t, \bar x+) = a^- \, q (t, \bar x-)
      \\
      P (\rho,q) (t, \bar x+) -P (\rho,q) (t, \bar x-)
      =
      \Xi_2 \left(a^+, a^-, (\rho,q)(t, \bar x-) \right) \,.
    \end{array}
  \right.
\end{equation}
The former relation in~\eqref{eq:52} ensures the conservation of mass
and fits in the framework of Section~\ref{sec:General} setting in the
first component of~\eqref{eq:DefPsiCoupling}
\begin{equation}
  \label{eq:54}
  \Xi_1 \left(a^+, a^-, (\rho^-,q^-)\right)
  =
  \left(\dfrac{a^-}{a^+} - 1\right) q^- \,.
\end{equation}
The literature offers a wide range of justifications, often
phenomenological, for specific choices of the function $\Xi_2$
in~\eqref{eq:52}, see for instance~\cite{ColomboGaravello2020,
  ColomboMarcellini2010, GuerraMarcelliniSchleper2009}. Note that, as
soon as $\Xi_2$ is of class $\C2$ in all variables, with
$\Xi_2 \left(a,a,(\rho,q)\right)=0$, and $a$ is in
$\BV (\reali; \reali)$, then Theorem~\ref{thm:main} applies ensuring
the existence of solutions to
\begin{equation}
  \label{eq:53}
  \left\{
    \begin{array}{rcl}
      \partial_t \rho
      + \partial_x q
      & =
      &  \displaystyle
        \sum_{x \in \mathcal I}
        \left(\dfrac{a (x-)}{a (x+)} - 1\right) q (t, x-) \, \delta_x
        - \dfrac{1}{a (x)} \, q (t,x) \, \mu
      \\
      \partial_t q
      + \partial_x P (\rho,q)
      &
        =
      & \displaystyle
        \sum_{x \in \mathcal I}
        \Xi_2\left(a (x+), a (x-), (\rho,q) (t,x-)\right)
        \delta_{x}
      \\
      &
      & \displaystyle
        +
        \partial_1 \Xi_2\left(a (x), a (x), (\rho,q) (t,x)\right) \,
        \mu
      \\

    \end{array}
  \right.
\end{equation}
where $\mathcal{I}$ is the set of points of discontinuity of $a$ and,
as soon as $a$ is smooth, $\mu$ has density $\partial_x a (x)$ with
respect to the Lebesgue measure. In~\eqref{eq:53} we also
used~\eqref{eq:55}.

As an application of Theorem~\ref{thm:main}, we characterize the class
of conditions $\Xi$ that yield in the limit the case of the smooth
pipe, i.e., equation~\eqref{eq:51}.

\begin{theorem}
  \label{thm:VarSec}
  Let $p$ satisfy~\ref{it:p}, $(\bar\rho,\bar q)$ be a subsonic state
  and $\bar a$ be positive. For any $\Xi_2$ of class $\C2$ with
  $\Xi_2 \left(a,a,(\rho,q)\right)=0$ and
  \begin{equation}
    \label{eq:56}
    \partial_1 \Xi_2 \left(a, a, (\rho,q)\right)
    = -\dfrac{1}{a} \, \dfrac{q^2}{\rho}
  \end{equation}
  there exists a positive $\delta$ such that for all initial data
  $(\rho_o, q_o)$ and for all $a \in \BV (\reali; \reali)$ with
  $a' \in \L1 (\reali; \reali)$ and
  \begin{displaymath}
    \norma{(\rho_o,q_o) - (\bar\rho, \bar q)}_{\L\infty (\reali; \reali^2)} < \delta
    \,,\quad
    \tv (\rho_o, q_o) < \delta
    \,,\quad
    \norma{a-\bar a}_{\L\infty (\reali, \reali)} < \delta
    \,,\quad
    \norma{a'}_{\L1 (\reali; \reali)} < \delta
  \end{displaymath}
  problem~\eqref{eq:51} admits a solution $(\rho_*,q_*)$. Moreover,
  there exists a sequence of piecewise constant approximations $a^h$
  of $a$, with $\tv (a^h) < \delta$, such that the corresponding
  solution $(\rho^h, q^h)$ converges to $(\rho_*, q_*)$ pointwise in
  time and in $\Lloc1$ in space. In particular, at each discontinuity
  point $y$ of $a^h$, $(\rho^h, q^h)$ satisfies the junction
  condition~\eqref{eq:52}.
\end{theorem}

\noindent The proof is deferred to~\S~\ref{sec:Applications}.

We now test the above condition against various junction condition
found in the literature, we refer in particular
to~\cite{ColomboGaravello2020} for the motivations and further
information of the conditions considered below. More precisely, with
reference to the labelling in Table~\ref{tab:A}, we consider
definition~\textbf{[L]} from~\cite{ColomboGaravello2007},
condition~\textbf{[p]} from~\cite{BandaHertyKlar1, BandaHertyKlar2},
condition~\textbf{[P]} from~\cite{ColomboGaravello2007,
  ColomboGaravello3} and condition~\textbf{[S]}
from~\cite{ColomboMarcellini2010, GuerraMarcelliniSchleper2009}. All
these conditions differ only in the second component $\Xi_2$, the
first one being fixed as in~\eqref{eq:54} to comply with mass
conservation.

Simple computations lead to the results in Table~\ref{tab:A}, where
the map $a \to R(a; a^-, \rho^-, q^-)$ is the first component of
the solution to the stationary version of~\eqref{eq:51}, parametrized
by the section $a$, i.e.,%
\begin{displaymath}
  \left\{
    \begin{array}{l@{\qquad\qquad}l}
      \dfrac{\d{}~}{\d{a}} q = -\dfrac{1}{a} \, q
      & \rho (a^-) = \rho^-
      \\
      \dfrac{\d{~}}{\d{a}} \left(P (\rho,q)\right) = -\dfrac{1}{a} \, \dfrac{q^2}{\rho}
      & q (a^-) = q^- \,.
    \end{array}
  \right.
\end{displaymath}%
\begin{table}
  \small
  \begin{tabular}{|c|c|c|c|}
    \hline
    & $\Psi_2(a^-,(\rho^-,q^-),a^+,(\rho^+,q^+))$
    & $\Xi_2 (a^+,a^-,(\rho^-,q^-))$
    & $D_1 \Xi_2 (a,a,(\rho,q))$
    \\
    \hline
    \textbf{[L]}
    & $a^+ P(\rho^+,q^+) - a^- P(\rho^-,q^-)$
    & $\left(\dfrac{a^-}{a^+}-1\right)^{\vphantom{\big|}}
      \left(\dfrac{(q^-)^2}{\rho^-} + p (\rho^-)\right)$
    & $-\dfrac{1}{a} \left(\dfrac{q^2}{\rho} + p (\rho)\right)$
    \\[15pt]
    \hline
    \textbf{[p]}
    & $p(\rho^+) - p(\rho^-)$
    & $\left(\left(\dfrac{a^-}{a^+}\right)^2-1\right)
      \dfrac{(q^-)^2}{\rho^-}^{\vphantom{\big|}}$
    & $ -\dfrac{2}{a} \,\dfrac{q^2}{\rho}$
    \\[15pt]
    \hline
    \textbf{[P]}
    & $P(\rho^+,q^+) - P(\rho^-,q^-)$
    & $0$
    & $0$
    \\[3pt]
    \hline
    \textbf{[S]}
    & $
      \begin{array}{l}
        \displaystyle
        \!\!
        a^+ P(\rho^+,q^+) - a^- P(\rho^-,q^-)
        \\
        \displaystyle
        -
        \int_{a^-}^{a^+} p \left( R(\alpha;a^-, \rho^-,q^-) \right) \, d\alpha\!\!
      \end{array}
    $
    & $
      \begin{array}{l}
        \!\!
        \left(\dfrac{a^-}{a^+}-1\right)^{\vphantom{\big|}}
        \left(\dfrac{(q^-)^2}{\rho^-} + p (\rho^-)\right)
        \\
        \displaystyle
        + \dfrac{1}{a^+} \!
        \int_{a^-}^{a^+} \!\!\! p\!\left(R (\alpha; a^-, \rho^-, q^-)\right)\d\alpha
        \\[10pt]
      \end{array}
    $
    & $-\dfrac{1}{a} \,\dfrac{q^2}{\rho}$
    \\
    \hline
  \end{tabular}
  \caption{Various definitions of junction conditions, with the
    corresponding functions $\Psi_2$
    from~\eqref{eq:IntroSingleCoupling}, $\Xi_2$
    from~\eqref{eq:DefPsiCoupling} and its partial derivative
    $\partial_1 \Xi_2$.}
  \label{tab:A}
\end{table}%
On the basis of Theorem~\ref{thm:VarSec}, we know that
condition~\textbf{[S]} is compatible with the smooth
limit~\eqref{eq:51}. Moreover, Theorem~\ref{thm:main} and
Table~\ref{tab:A}, in particular the comparison of the rightmost
column with~\eqref{eq:56}, ensure that all the other conditions do
\emph{not} converge to~\eqref{eq:51} in the smooth pipe limit.

Remark that substituting in~\textbf{[S]} any other smooth function
$R = R (a; a^-, \rho^-, q^-)$ such that
$R (a^-; a^-, \rho^-, q^-) = \rho^-$ yields a new condition at the
junction compatible with the smooth limit.

\subsection{Full Gas Dynamics in Pipes with Varying Section}
\label{sec:full-gas-dynamics}

The full Euler system in a pipeline with smoothly varying section
$a = a (x)$ is
\begin{equation}
  \label{eq:64}
  \left\{
    \begin{array}{l}
      \partial_t\rho + \partial_x (\rho\, v)
      = - \dfrac{a'}{a} \, \rho \, v
      \\
      \partial_t (\rho\, v) + \partial_x \left(\rho \, v^2 + p \right)
      = - \dfrac{a'}{a} \, \rho \, v^2
      \\
      \partial_t \left(\frac{1}{2}\, \rho \, v^2 + \rho \, e\right)
      +
      \partial_x\left(
      v \left(\frac{1}{2}\, \rho \, v^2 + \rho \, e + p\right)
      \right)
      = - \dfrac{a'}{a} \, v
      \left(\frac{1}{2}\, \rho \, v^2 + \rho \, e+ p\right)
    \end{array}
  \right.
\end{equation}
see, for instance~\cite{ColomboMarcellini2010Euler,
  GuerraMarcelliniSchleper2009, TPLiu1979, Liu1982}. Here, $x$ is the
abscissa along the pipe, $\rho > 0$ denotes gas density, $q\in\reali$
the momentum density, $p = p(\rho,s)$ the pressure and $e=e (\rho,s)$
the energy density and $s$ the entropy density. These two latter
functions satisfy
\begin{enumerate}[label={\bf{(E)}}]
\item\label{it:E}
  $e \in \C2(\left]0, +\infty\right[ \times \reali, \left]0,
    +\infty\right[)$ and $\partial_s e (\rho, s) > 0$ for all $\rho>0$
  and $s \in \reali$.
\end{enumerate}
\begin{enumerate}[label={\bf{(P)}}]
\item\label{it:P}
  $p \in \C2(\left]0, +\infty\right[ \times \reali, \left]0,
    +\infty\right[)$,
  $p (\rho, s) = \rho^2 \, \partial_\rho e (\rho, s)$,
  $\partial_\rho p (\rho, s) > 0$ and
  $\partial^2_{\rho\rho} \left(\rho \, p(\rho,s)\right) > 0$ for all
  $\rho>0$ and $s \in \reali$.
\end{enumerate}

\noindent We restrict our attention to the subsonic region where
$v \in \mathopen]0, \sqrt{\partial_\rho p (\rho,s)}\mathclose[$.

The conditions found in the literature,
see~\cite{ColomboMarcellini2010Euler}, imposed at a point $\bar x$
where the section suffers a discontinuity fit into the form
\begin{equation}
  \label{eq:65}
  \left\{
    \begin{array}{l}
      \Delta (a \, \rho\, v) (t, \bar x)= 0
      \\
      \Delta (\rho\, v^2 + p) (t, \bar x)
      =
      \Xi_2 \left(a (x+), a (x-), (\rho, v, s) (t, x-)\right)
      \\
      \Delta \left(
      a \, v \left(\frac{1}{2}\, \rho \, v^2 + \rho \, e + p\right)
      \right) (t, \bar x)
      =
      0
    \end{array}
  \right.
\end{equation}
The conservation of mass imposed by the first equality and the
conservation of energy imposed by the third equality in~\eqref{eq:65}
amount to setting
\begin{displaymath}
  \begin{array}{rcl}
    \Xi_1 \left(a^+, a^-, (\rho^-,v^-,s^-)\right)
    & =
    & \left(\dfrac{a^-}{a^+} - 1\right) \rho^- \, v^-
    \\[12pt]
    \Xi_3 (a^+, a^-, u^-)
    & =
    & \left(\dfrac{a^-}{a^+} - 1 \right)
      \left(
      v^- \left(\frac{1}{2}\, \rho^-  (v^-)^2 + \rho^- \, e^- + p^-\right)
      \right)
    \\
  \end{array}
\end{displaymath}
so that
\begin{displaymath}
  \begin{array}{rcl}
    \partial_1 \Xi_1 \left(a, a, (\rho,v,s)\right)
    & =
    & -\dfrac{1}{a} \, \rho \, v
    \\[3pt]
    \partial_1 \Xi_3 (a,a,u)
    & =
    & -\dfrac{1}{a}
      \left(
      v \left(\frac{1}{2}\, \rho \, v^2 + \rho \, e + p\right)
      \right)
  \end{array}
\end{displaymath}
The second equality in~\eqref{eq:65} is treated in different ways in
the literature, giving rise to conditions analogous to those
considered in \S~\ref{sec:isentr-gas-pipe}. Indeed, Table~\ref{tab:A}
directly extends to the present full $3\times3$ case, simply
understanding the map $R$ as the $\rho$ component
$a \to R (a; a^-, \rho^-, v^-, s^-)$ in the solution to the stationary
Cauchy Problem
\begin{equation}
  \label{eq:67}  \left\{
    \begin{array}{l@{\qquad\qquad}l}
      \dfrac{\d{}}{\d{a}} (\rho\, v)
      = - \dfrac{1}{a} \, \rho \, v
      &\rho (a^-) = \rho^-
      \\[8pt]
      \dfrac{\d{}}{\d{a}} \left(\rho \, v^2 + p \right)
      = - \dfrac{1}{a} \, \rho \, v^2
      & v (a^-) = v^-
      \\[8pt]
      \dfrac{\d{}}{\d{a}} \left(
      v \left(\frac{1}{2}\, \rho \, v^2 + \rho \, e + p\right)
      \right)
      = - \dfrac{1}{a} \, v
      \left(\frac{1}{2}\, \rho \, v^2 + \rho \, e+ p\right)
      & s (a^-) = s^- \,.
    \end{array}
  \right.
\end{equation}

\subsection{Balance Laws with Measure Valued Source Term}

The theory developed in Section~\ref{sec:General} allows to give a
meaning to the following balance law, where the source term is
\emph{non conservative}:
\begin{equation}
  \label{eq:59}
  \partial_t u +\partial_x f(u)
  =
  \partial_\zeta G(\zeta,u) \, D\zeta
\end{equation}
where $G$ is smooth and $\zeta$ has bounded variation. In the case $G$
independent of $u$, we recover the \emph{conservative}
case~\eqref{eq:57}.  In the general, \emph{non} conservative
case,~\eqref{eq:59} can be given different meanings.

A choice consists in setting
\begin{equation}
  \label{eq:58}
  \Xi (z^+, z^-,u^-)
  =
  G (z^+,u^-) - G (z^-,u^-) \,,
\end{equation}
corresponding to the following condition at each point of jump:
\begin{displaymath}
  f (u^+) - f (u^-)
  =
  G (z^+, u^-) - G (z^-, u^-) \,.
\end{displaymath}

The framework developed in the preceding section in connection with
the Cauchy Problem~\eqref{eq:21} comprises~\eqref{eq:59}. Therefore,
we can particularize Definition~\ref{def:sol} to the general case of
non conservative products of the type~\eqref{eq:59}.

\begin{definition}
  \label{def:NonCons}
  Fix an initial datum $u_o \in \Lloc1 (\reali; \reali^n)$.  Let
  $\Xi \in \C2 (\mathcal{Z}\times\mathcal{Z}\times\Omega; \reali^n)$
  be such that
  \begin{equation}
    \label{eq:63}
    D_1 \Xi (z, z, u) = D_z G (z, u) \,.
  \end{equation}
  Then, a map $u \in \C0 ([0,T]; \Lloc1 (\reali; \reali^n))$ with
  $u (t) \in \BV (\reali; \reali^n)$ and left continuous for all
  $t \in [0,T]$, is a \emph{$\Xi$--solution} to~\eqref{eq:21} if for
  all test function
  $\phi \in \Cc1 (\left]0, T\right[ \times \reali; \reali)$,
  \begin{eqnarray}
    \nonumber
    &
    & -\int_0^{+\infty} \int_\Omega
      \left(
      u (t,x) \, \partial_t \phi (t,x)
      +
      f\left(u (t,x)\right) \, \partial_x \phi (t,x)
      \right)
      \d{x} \d{t}
    \\
    \label{eq:61}
    & =
    & \sum_{x \in \mathcal{I}}
      \int_0^{+\infty}
      \Xi\left(\zeta (x+), \zeta (x), u (t, x)\right)
      \phi (t,x) \d{t}
    \\
    \nonumber
    &
    & +
      \int_0^{+\infty} \int_{\reali}
      D_z G\left(\zeta (x), u (t,x)\right)
      \phi (t,x) \, D\mu(x)\, \d{t}
  \end{eqnarray}
  where $\mathcal{I}$ is the set of jump points of $\zeta$ and $\mu$
  is as in~\eqref{eq:20}, and moreover $u (0) = u_o$.
\end{definition}
This definition clearly separates those part of the solution that
depend exclusively on~\eqref{eq:59} from those part, in the middle
term in~\eqref{eq:61}, that depend on the arbitrary choice of $\Xi$.

In particular, the choice~\eqref{eq:58} yields
\begin{equation}
  \label{eq:69}
  \Xi \left(\zeta (x+), \zeta (x),u (t,x)\right)
  =
  G \left(\zeta (x+),u (t, x)\right)
  -
  G \left(\zeta (x),u (t,x)  \right)
\end{equation}
where we keep using the left continuous representatives. For
completeness, we remark that the alternative choice
$ \Xi (z^+, z^-,u^-) = G (z^+,u^+) - G (z^-,u^+)$ also meets
condition~\eqref{eq:63}.

A straightforward application of Theorem~\ref{thm:main} now ensures
the existence of $\Xi$--solutions to~\eqref{eq:59}, as soon as
$G \in \C2 (\mathcal{Z}\times \Omega; \reali^{n\times m})$,
$\zeta \in \BV (\reali; \mathcal{Z})$,
$\Xi \in \C2 (\mathcal{Z} \times \mathcal{Z} \times\Omega; \reali^m)$
and satisfies~\eqref{eq:63}. Moreover, these solutions are limits of
\emph{``discretized''} approximations where~\eqref{eq:DefPsiCoupling}
is imposed to the points of jump in $\zeta$.

\section{Technical Details}

Below, by $\O$ we denote a constant depending exclusively on $f$ and
$\Xi$.

\subsection{Preliminary Results}
\label{sub:GerneralSingle}

First, we prove a Lipschitz-type estimate on the map $\Xi$ which we
use throughout this paper.

\begin{lemma}
  \label{lem:EstimateXi2}
  Assume that~\ref{eq:Xi1}, \ref{eq:Xi3} hold. Then,
  \begin{equation}
    \label{eq:lipXi}
    \left\|\Xi(z^{+},z^{-},u_{2}) - \Xi(z^{+},z^{-},u_{1})\right\|
    =
    \O \, \norma{z^{+}-z^{-}} \, \norma{u_{2}-u_{1}} \, .
  \end{equation}
\end{lemma}

\begin{proof}
  Since the map $u\mapsto \Xi(z^{+},z^{-},u)$ is smooth, we can
  compute
  \begin{displaymath}
    \begin{split}
      &\left\|\Xi(z^{+},z^{-},u_{2}) -
        \Xi(z^{+},z^{-},u_{1})\right\|\\
      &\quad\le \left\|u_{2}-u_{1}\right\|
      \int_{0}^{1}\left\|D_{u}\Xi(z^{+},z^{-},u_{1}+s\left(u_{2}-u_{1}\right))
        -D_{u}\Xi(z^{-},z^{-},u_{1}+s\left(u_{2}-u_{1}\right))\right\|\d{s}\\
      &\quad \le \O \left\|z^{+}-z^{-}\right\| \,
      \left\|u_{2}-u_{1}\right\|,
    \end{split}
  \end{displaymath}
  where we used the equality
  $D_{u}\Xi(z^{-},z^{-},u_{1}+s\left(u_{2}-u_{1}\right))=0$.
\end{proof}

Introduce a map $T$ related to the generalized Riemann problem.

\begin{lemma}
  \label{lem:ExistenceT}
  Let $f$ satisfy~\ref{it:f1}--\ref{it:f3}, $\Xi$
  satisfy~\ref{eq:Xi1}, \ref{eq:Xi3} and $A_{i_o}$ be as
  in~\eqref{eq:23}. Then, for any $\bar z \in \mathcal{Z}$ and
  $\bar u \in A_{i_o}$, there exists $\bar\delta > 0$ and a Lipschitz
  map
  \begin{equation*}
    T \colon B(\bar z;\bar\delta)^2 \times B(\bar u;\bar \delta) \to A_{i_o}
  \end{equation*}
  such that
  \begin{equation*}
    \begin{cases}
      f (u^+) - f (u^-) = \Xi (z^+,z^-,u^-)
      \\
      z^+, z^-\in B(\bar z;\bar\delta)
      \\
      u^+,u^-\in B(\bar u;\bar\delta)
    \end{cases}
    \Longleftrightarrow \quad u^+=T(z^+, z^-,u^-) \, .
  \end{equation*}
  Furthermore,
  \begin{eqnarray}
    \label{eq:37}
    \norma{T(z^+, z^-,u^-)-u^-}
    & =
    & \O  \norma{z^+-z^-} \, ,
    \\
    \label{eq:38}
    \norma{T(z^+, z^-,u_2)-T(z^+, z^-,u_1)-(u_2-u_1)}
    & =
    & \O \norma{z^+-z^-} \, \norma{u_2-u_1} \, .
  \end{eqnarray}
\end{lemma}

\begin{proof}
  Since $\bar u \in A_{i_o}$, \ref{it:f1} and~\ref{it:f2} ensure that
  the function $f$ is locally invertible at $\bar u$. We define
  \begin{equation}
    \label{eq:defT}
    T (z^+, z^-, u^{-})
    =f^{-1}\left(f(u^{-}) + \Xi(z^{+},z^{-},u^{-})\right).
  \end{equation}
  By \ref{eq:Xi1}, \ref{eq:Xi3} we compute
  \begin{eqnarray*}
    \left\|T(z^+, z^-, u^{-})-u^{-}\right\|
    & =
    & \left\|T(z^+, z^-, u^{-})-f^{-1}\left(f(u^{-})\right)\right\|
    \\
    & =
    & \O
      \left\|\Xi(z^{+},z^{-},u^{-}) - \Xi(z^{-},z^{-},u^{-})\right\|
    \\
    & =
    & \O
      \left\|z^{+}-z^{-}\right\|,
  \end{eqnarray*}
  proving~\eqref{eq:37}.  Introduce the smooth map
  \begin{displaymath}
    b\left(\xi,\Delta,v\right)
    =
    f^{-1}\left(f(u_{1}+v)+\xi + \Delta\right)
    - f^{-1}\left(f(u_{1})+\xi\right) - v \,.
  \end{displaymath}
  Since $b\left(\xi,0,0\right)=b\left(0,0,v\right)=0$, the estimate
  \begin{displaymath}
    b\left(\xi,\Delta,v\right)
    =
    \O \left[\left\|\xi\right\| \cdot \left\|v\right\| + \left\|\Delta\right\|\right]
  \end{displaymath}
  holds, see~\cite[\S~2.9]{Bressan2000}. The left hand side
  of~\eqref{eq:38} can be written as
  \begin{displaymath}
    \begin{split}
      &\left\|T(z^+, z^-,u_2)-T(z^+, z^-,u_1)-(u_2-u_1)\right\|
      \\
      &\quad = \left\|b\left[\Xi(z^{+},z^{-},u_{1}),
          \Xi(z^{+},z^{-},u_{2}) -
          \Xi(z^{+},z^{-},u_{1}),u_{2}-u_{1}\right]\right\|
      \\
      &\quad\le \O \left[\left\|\Xi(z^{+},z^{-},u_{1})\right\|\cdot
        \left\|u_{2}-u_{1}\right\| +\left\|\Xi(z^{+},z^{-},u_{2}) -
          \Xi(z^{+},z^{-},u_{1})\right\|\right]
      \\
      & \quad \le \O \, \norma{z^+-z^-} \, \norma{u_2-u_1} \,.
    \end{split}
  \end{displaymath}
\end{proof}

\subsection{The Case $\zeta$ Piecewise Constant}
\label{sub:GeneralFinite}

In this section, we consider the case of $\mathcal{I}$ being finite,
with $\zeta$ being piecewise constant. We index the points
$x \in \mathcal{I}$ so that $x_i < x_j$ if and only if $i < j$.  In
this case, the general Definition~\ref{def:sol} reduces to the
following one, often found in the literature, see for
instance~\cite{ColomboHolden2016, ColomboMarcellini2010Euler,
  ColomboMarcellini2010, HoldenRisebro1999}.

\begin{definition}
  \label{def:WeakSolutionOneJunction}
  A weak solution to the Cauchy Problem~(\ref{eq:21}) with a piecewise
  constant $\zeta$ is a map
  $u \in \C0(\left[0, +\infty\right[; \Lloc1(\reali;\Omega))$ with
  $u(t) \in \BV(\reali;\Omega)$, left continuous, for all $t \geq 0$,
  such that for all
  $\phi \in \Cc1(\left]0,+\infty\right[ \times\reali; \reali)$ whose
  support does not intersect
  $\left[0,+\infty\right[\times \mathcal{I}$
  \begin{equation}
    \label{eq:27}
    \int_0^{+\infty} \int_{\reali}
    \left(u \, \partial_t \phi + f(u) \, \partial_x\phi\right) \d x \d t
    = 0 \,,
  \end{equation}
  $u (0) = u_o$ and for all $x \in \mathcal{I}$
  \begin{displaymath}
    f \left(u (t, x+)\right) - f\left(u (t, x)\right) =
    \Xi \left(\zeta (x+),\zeta (x),u(t,x)\right)
    \quad \mbox{for a.e. } t \in \mathopen[0, +\infty\mathclose[ \,.
  \end{displaymath}
\end{definition}

\subsubsection{The Generalized Riemann Problem}
\label{sub:GRP}

By \emph{Generalized Riemann Problem} we consider the Cauchy
Problem~(\ref{eq:21}) with $\zeta$ and the initial datum $u_o$ being
as follows:
\begin{equation}
  \label{eq:25}
  \zeta (x)
  =
  z^- \, \caratt{\left]-\infty, 0\right[} (x)
  +
  z^+ \, \caratt{\left]0, +\infty\right[} (x)
  \quad \mbox{ and } \quad
  u_o (x)
  =
  u^\ell \, \caratt{\left]-\infty, 0\right[} (x)
  +
  u^r \, \caratt{\left]0, +\infty\right[} (x) \,.
\end{equation}

For $u \in A_{i_o}$, call $\sigma_i \to H_i (\sigma_i) (u)$ the Lax
curve of the $i$-th family exiting $u$, see~\cite[\S~5.2]{Bressan2000}
or~\cite[\S~9.3]{Dafermos2000}. Introduce recursively the states
$w_0, \ldots, w_{n+1}$ with $w_0 = u^\ell,\, w_{n+1} = u^r$ and
\begin{displaymath}
  \left\{
    \begin{array}{r@{\,}c@{\,}l@{\qquad}r@{\,}c@{\,}l}
      w_{i+1}
      & =
      & H_{i+1}(\sigma_{i+1})(w_i)
      & \mbox{ if }i
      & =
      & 0, \ldots, i_o-1,
      \\
      w_{i_o+1}
      & =
      & T(z^+, z^-, w_{i_o})
      \\
      w_{i+1}
      & =
      & H_{i}(\sigma_{i})(w_i)
      & \mbox{ if }i
      & =
      &i_o+1, \ldots, n\,.
    \end{array}
  \right.
\end{displaymath}
If $z^{+}-z^{-}$ is sufficiently small,
\cite[Lemma~3]{AmadoriGosseGuerra2002} ensures that the waves' sizes
$(\sigma_1, \ldots, \sigma_n)$ and the states $(w_1, \ldots, w_n)$
exist, are uniquely defined and are Lipschitz continuous functions of
$z^-, z^+, u^\ell, u^r$, which ensures also the well posedness of the
Generalized Riemann Problem~(\ref{eq:21}),~\eqref{eq:25}. The
following notation is of use below:
\begin{equation}
  \label{eq:39}
  (\sigma_1, \ldots, \sigma_n)
  =
  E(z^+, z^-, u^r,  u^\ell) \,.
\end{equation}
We thus write the solution $u$ to the Generalized Riemann
Problem~(\ref{eq:21})~\eqref{eq:25}, in the sense of
Definition~\ref{def:WeakSolutionOneJunction}, as the glueing along
$x = 0$ of the Lax solutions to the (standard) Riemann Problems
\begin{displaymath}
  \left\{
    \begin{array}{l}
      \partial_t u + \partial_x f (u) = 0
      \\
      u (0,x)
      = u^\ell \caratt{\left]-\infty, 0\right[} (x)
      +
      w^{i_o} \caratt{\left]0, +\infty\right[} (x),
    \end{array}
  \right.
  \quad
  \left\{
    \begin{array}{l@{}}
      \partial_t u + \partial_x f (u) = 0
      \\
      u (0,x)
      = w^{i_o+1} \caratt{\left]-\infty, 0\right[} (x)
      +
      u^r \caratt{\left]0, +\infty\right[} (x).
    \end{array}
  \right.
\end{displaymath}

\subsubsection{Interaction Estimates}

Let $u^\ell, u^r \in \Omega$ be initial states for the Generalized
Riemann Problem~(\ref{eq:21})~\eqref{eq:25}. We separate the waves
with negative or positive propagation speed as follows:
\begin{align}
  \begin{split}
    \pmb\sigma'=(\sigma_1,\dots,\sigma_{i_o},0,\dots,0),&\qquad \pmb\sigma''=(0,\dots,0,\sigma_{i_o+1},\dots,\sigma_n),\\
    \pmb\sigma &=\pmb\sigma'+\pmb\sigma''\in\reali^n.
  \end{split}
\end{align}
Given two $n$-tuples of waves $\pmb{\alpha}$ and $\pmb{\beta}$, the
waves $i$ with size $\alpha_i$ and $j$ with size $\beta_j$ are
\emph{approaching} whenever $i>j$ or
$\min\left\{\alpha_i, \beta_j\right\} < 0$. Call
$\mathcal{A}_{\pmb{\alpha},\pmb{\beta}}$ the set of these pairs.

In the following we recall several lemmas which are straightforward
generalizations of results in~\cite{AmadoriGosseGuerra2002}.
\begin{lemma}
  \label{lem:AGG4}
  Let $f$ satisfy~\ref{it:f1}--\ref{it:f3}, $\Xi$
  satisfy~\ref{eq:Xi1}, \ref{eq:Xi3} and $A_{i_o}$ be as
  in~\eqref{eq:23}. Fix $z^+, z^- \in \mathcal{Z}$ and
  $u^\ell, u^r\in A_{i_o}$. Then, there exists a $\delta>0$ such that
  if $z^+, z^-\in B(\bar z;\delta)$,
  $u^\ell, u^r \in B(\bar u; \delta)$, we have
  \begin{align*}
    \norma{u^r - u^\ell}
    & = \O \left(\norma{\pmb\sigma} + \norma{z^+-z^-}\right),
    \\
    \norma{\pmb\sigma}
    & = \O \left(\norma{u^r -  u^\ell} + \norma{z^+-z^-} \right).
  \end{align*}
\end{lemma}

\begin{proof}
  By Lemma~\ref{lem:ExistenceT}, we get
  \begin{displaymath}
    \norma{u^r - u^\ell}
    \le \sum_{i=1}^{n+1} \norma{w_i - w_{i-1}}
    = \O \norma{\pmb{\sigma}} {+} \norma{T(z^+, z^-,w_{i_o})-w_{i_o}}
    = \O \! \left( \norma{\pmb{\sigma}} {+} \norma{z^+-z^-} \right) \!.
  \end{displaymath}
  By the Lipschitz continuity of $E$ as defined in~\eqref{eq:39}, we
  get
  \begin{align*}
    \norma{\pmb\sigma}
    & = \norma{E(z^+, z^-, u^r, u^\ell)
      -
      E(z^+, z^-, T(z^+, z^-,u^\ell), u^\ell)}
    \\
    & = \O \norma{u^r-T(z^+, z^-,u^\ell)}
    \\
    & = \O \left(
      \norma{u^r -u^\ell} + \norma{u^\ell -T(z^+, z^-,u^\ell)} \right)
    \\
    & = \O \left( \norma{u^r -u^\ell} + \norma{z^+ - z^-} \right)
  \end{align*}
  completing the proof.
\end{proof}

\begin{lemma}[{\cite[Lemma~5]{AmadoriGosseGuerra2002}}]
  \label{lem:AGG5}
  Let $f$ satisfy~\ref{it:f1}--\ref{it:f3} and $A_{i_o}$ be as
  in~\eqref{eq:23}. For $u \in \Omega$ sufficiently close to
  $\bar u\in A_{i_o}$ and $y_1,y_2,\pmb\alpha \in \reali^n$
  sufficiently small, we have
  \begin{equation*}
    \norma{y_2 + H(\pmb\alpha)(u) - H(\pmb\alpha)(u+y_1)}
    =
    \O \left(
      \norma{\pmb\alpha} \, \norma{y_1} + \norma{y_1-y_2}
    \right).
  \end{equation*}
\end{lemma}

\begin{lemma}
  \label{lem:AGG6}
  Let $f$ satisfy~\ref{it:f1}--\ref{it:f3}, $\Xi$
  satisfy~\ref{eq:Xi1}, \ref{eq:Xi3} and $A_{i_o}$ be as
  in~\eqref{eq:23}. For $u \in \Omega$ sufficiently close to
  $\bar u \in A_{i_o}$, $z^+, z^-\in \mathcal Z$ sufficiently close to
  $\bar z \in \mathcal{Z}$ and $\pmb\alpha \in \reali^n$ sufficiently
  small, we have
  \begin{equation*}
    \norma{T(z^+, z^-,H(\pmb\alpha)(u))-H(\pmb\alpha)(T(z^+, z^-,u))}
    =
    \O \, \norma{\pmb\alpha} \, \norma{z^+-z^-} \,.
  \end{equation*}
\end{lemma}

\begin{proof}
  Applying Lemma~\ref{lem:AGG5} with $y_1 = T(z^+, z^-,u)-u$ and
  $y_2 = T(z^+, z^-,H(\pmb\alpha)(u))-H(\pmb\alpha )(u)$ gives
  \begin{align*}
    & \norma{T(z^+, z^-,H(\pmb\alpha)(u)) - H(\pmb\alpha)(T(z^+, z^-,u))}
    \\
    \le \;
    & \O \left(
      \norma{\pmb\alpha} \norma{T(z^+, z^-,u)-u}
      + \norma{T(z^+, z^-,H(\pmb\alpha)(u))-H(\pmb\alpha )(u)-T(z^+, z^-,u)+u}
      \right).
  \end{align*}
  The result follows from Lemma~\ref{lem:ExistenceT}.
\end{proof}

\begin{lemma}
  \label{lem:AGG7}
  Let $f$ satisfy~\ref{it:f1}--\ref{it:f3}, $\Xi$
  satisfy~\ref{eq:Xi1}, \ref{eq:Xi3} and $A_{i_o}$ be as
  in~\eqref{eq:23}. Fix $z^+, z^- \in \mathcal{Z}$ and
  $u^\ell, u^r\in A_{i_o}$. Then, there exists a $\delta>0$ such that
  if $ u^r, u^\ell\in B(\bar u;\delta)$ and
  $z^+, z^-\in B(\bar z;\delta)$. Let
  \begin{eqnarray*}
    u^-
    & =
    & H(\pmb\alpha)(u^\ell),
    \\
    u^r
    & =
    & H(\pmb\beta'')\left(T\left(z^+, z^-,H(\pmb\beta')(u^-)\right)\right),
    \\
    u^r
    & =
    & H(\pmb\sigma'')\left(T\left(z^+, z^-,H(\pmb\sigma')(u^\ell)\right)\right),
  \end{eqnarray*}
  with $\pmb\alpha,\pmb\beta,\pmb\sigma\in\reali^n$ and
  $u^-\in B(\bar u;\delta)$. Then,
  \begin{equation}
    \label{eq:LemmaAGG7part1}
    \sum_{i=1}^n \modulo{\sigma_i-\alpha_i-\beta_i}
    =
    \O \left(
      \sum_{(i,j) \in \mathcal{A}_{\pmb{\alpha},\pmb{\beta}}}
      \modulo{\alpha_i \, \beta_j}
      +
      \norma{z^+-z^-} \, \sum_{i>i_o} \modulo{\alpha_i}
    \right),
  \end{equation}
  where $\mathcal A_{\pmb{\alpha},\pmb{\beta}}$, as above, denotes the
  set of approaching waves. Analogously, if
  \begin{eqnarray*}
    u^+
    & =
    & H(\pmb{\alpha''})\left(T\left(z^+,z^-,H(\pmb{\alpha'})(u^\ell)\right)\right),
    \\
    u^r
    & =
    & H(\pmb{\beta})(u^+),
    \\
    u^r
    & =
    & H(\pmb\sigma'')\left(T\left(z^+, z^-,H(\pmb\sigma')(u^\ell)\right)\right),
  \end{eqnarray*}
  then,
  \begin{equation}
    \label{eq:LemmaAGG7part2}
    \sum_{i=1}^n \modulo{\sigma_i-\alpha_i-\beta_i}
    =
    \O \left(
      \sum_{(i,j) \in \mathcal{A}_{\pmb{\alpha},\pmb{\beta}}}
      \modulo{\alpha_i\beta_j}
      +
      \norma{z^+-z^-} \sum_{i<i_o} \modulo{\beta_i}
    \right).
  \end{equation}
\end{lemma}

\begin{proof}
  It is sufficient to prove~\eqref{eq:LemmaAGG7part1},
  since~\eqref{eq:LemmaAGG7part2} is proved analogously. We set
  \begin{eqnarray*}
    \tilde{u}
    & =
    & H(\pmb{\alpha''} + \pmb{\beta''})
      \left(
      T \left(z^+, z^-,H(\pmb{\alpha'} + \pmb{\beta'})(u^\ell) \right)
      \right),
    \\
    u_1
    & =
    & H(\pmb\beta'') \circ H(\pmb\alpha'')
      \left(T\left(z^+, z^-,H(\pmb\alpha')(u^\ell)\right)\right),
    \\
    u_2
    & =
    & H(\pmb\beta'')\left(T\left(z^+, z^-,H(\pmb\alpha)(u^\ell)\right)\right).
  \end{eqnarray*}
  By the Lipschitz continuity of $E$, we obtain
  \begin{eqnarray*}
    \norma{\pmb\sigma-(\pmb\alpha+\pmb\beta)}
    & =
    & \norma{E(z^+, z^-,u^r,u^\ell)-E(z^+, z^-,\tilde u,u^\ell)}
    \\
    & =
    & \O \, \norma{u^r - \tilde u}
    \\
    & =
    & \O \left(
      \norma{u^r - \tilde u  + u_1 - u_2}
      +
      \norma{u_1 - u_2}
      \right).
  \end{eqnarray*}
  To estimate the first term we consider the function
  $u^r - \tilde u + u_1 - u_2$ which is $\C2$
  w.r.t.~$\pmb{\alpha}, \pmb{\beta}$. Moreover, we assume that there
  are no approaching waves and obtain
  \begin{displaymath}
    \pmb{\alpha} = (\alpha_1,\dots,\alpha_i,0,\dots,0)\,,\quad
    \pmb{\beta} = (0,\dots,0,\beta_j,\dots,\beta_n)\,,\quad i\le j.
  \end{displaymath}
  The case $i=j$ is given in the case of two rarefaction waves and
  $\alpha_i,\beta_i\ge 0$. If $i\le i_o$, then $\pmb\alpha''=0$,
  $H(\pmb{\alpha'} + \pmb{\beta'})(u) = H(\pmb{\beta'}) \circ
  H(\pmb{\alpha'})(u)$, whence $u^r = \tilde{u}$ and $u_1 = u_2$.

  If $i > i_o$, then $\pmb\beta'=0$ and
  $H(\pmb{\alpha''} + \pmb{\beta''})(u) = H(\pmb{\beta''}) \circ
  H(\pmb{\alpha''})(u)$, whence $u^r = u_2$ and $\tilde{u} = u_1$.

  In all cases we get $u_0^+ - \tilde u^+ +u_1-u_2=0$. Standard
  considerations (see e.g.~\cite[\S~7.3]{Bressan2000},
  \cite[\S~13.3]{Dafermos2000} or~\cite{Yong1999}) and
  Lemma~\ref{lem:ExistenceT} lead to
  \begin{equation*}
    \norma{u^r - \tilde{u} + u_1 - u_2}
    =
    \O \sum_{(i,j) \in \mathcal{A}_{\pmb{\alpha}.\pmb{\beta}}}
    \modulo{\alpha_i \, \beta_j}
  \end{equation*}
  in the general case.

  Concerning $\norma{u_1-u_2}$, we get
  \begin{displaymath}
    \norma{u_1-u_2}
    \le
    \O \, \norma{H(\pmb{\alpha''}) \left(
        T(z^+, z^-,H(\pmb\alpha')(u^\ell))
      \right) -
      T(z^+, z^-,H(\pmb{\alpha})(u^\ell))} \,.
  \end{displaymath}
  The equality
  $H(\pmb\alpha)(u) = H(\pmb\alpha'')\circ H(\pmb\alpha')(u)$ and
  Lemma~\ref{lem:AGG6} with $u = H(\pmb{\alpha'})(u^\ell)$ lead to
  \begin{displaymath}
    \norma{u_1-u_2}
    =
    \O \; \norma{z^+ - z^-} \sum_{i>i_o} \; \modulo{\alpha_i} \,.
  \end{displaymath}
  The result follows.
\end{proof}

Lemma~\ref{lem:AGG7} suggests that the quantity $\norma{z^+-z^-}$ is a
convenient way to measure the strength of the zero--waves associated
to the coupling condition. More precisely, we define the strength of
the zero--wave at a junction with parameters
$z^+, z^- \in \mathcal{Z}$ as $\sigma = \norma{z^+-z^-}$.

\subsubsection*{Wave-front tracking approximate solutions}

We adapt the wave-front tracking techniques
from~\cite{AmadoriGosseGuerra2002, Bressan2000, ColomboMarcellini2010,
  GuerraMarcelliniSchleper2009} to construct a sequence of approximate
solutions to the Cauchy problem~\eqref{eq:21} and prove uniform
$\BV$-estimates in space. The approximate solutions converge towards a
solution to the Cauchy problem with finitely many junctions. First, we
define the approximations.

\begin{definition}
  \label{def:WFTApproxSol}
  Let $\zeta \in \BV (\reali; \mathcal{Z})$ be piecewise constant.
  For $\epsilon>0$, a continuous map
  \begin{equation*}
    u^\epsilon \colon \left[0, +\infty\right[ \to \Lloc1(\reali; \reali^n)
  \end{equation*}
  is an $\epsilon$-approximate solution to~\eqref{eq:21} if the
  following conditions hold:
  \begin{itemize}
  \item $u^\epsilon$ as a function of $(t,x)$ is piecewise constant
    with discontinuities along finitely many straight lines in the
    $(t,x)$-plane. There are only finitely many wave-front
    interactions and at most two waves interact with each other. There
    are four types of discontinuities: shocks (or contact
    discontinuities), rarefaction waves, non--physical waves and
    zero--waves. We distinguish these waves' indexes in the sets
    $\mathcal J =\mathcal{S} \cup \mathcal{R} \cup
    \mathcal{N}\mathcal{P} \cup \mathcal{Z}\mathcal{W}$, the generic
    index in $\mathcal{J}$ being $\alpha$.
  \item At a shock (or contact discontinuity) $x_\alpha=x_\alpha(t)$,
    $\alpha\in\mathcal S$, the traces $u^+=u^\epsilon (t,x_\alpha+)$
    and $u^-=u^\epsilon (t,x_\alpha-)$ are related by
    $u^+=H_{i_\alpha}(\sigma_\alpha)(u^-)$ for $1\le i_\alpha\le n$
    and wave-strength $\sigma_\alpha$. If the $i_\alpha$-th family is
    genuinely nonlinear, the Lax entropy condition $\sigma_\alpha<0$
    holds and
    \begin{equation*}
      \modulo{\dot x_\alpha -\lambda_{i_\alpha}(u^+,u^-)}
      \le \epsilon,
    \end{equation*}
    where $\lambda_{i_\alpha}(u^+,u^-)$ is the wave speed described by
    the Rankine-Hugoniot conditions.
  \item For a rarefaction wave $x_\alpha = x_\alpha(t)$,
    $\alpha\in\mathcal R$ the traces are related by
    $u^+ = H_{i_\alpha}(\sigma_\alpha)(u^-)$ for a genuinely nonlinear
    family $1 \le i_\alpha \le n$ and wave-strength
    $0 < \sigma_\alpha \le \epsilon$. Moreover,
    \begin{equation*}
      \modulo{\dot x_\alpha -\lambda_{i_\alpha}(u^+)} \le \epsilon.
    \end{equation*}
  \item All non--physical fronts $x = x_\alpha(t)$,
    $\alpha \in \mathcal{NP}$ travel at the same speed
    $\dot x_\alpha = \hat\lambda$ with
    $\hat\lambda > \sup_{u,i} \modulo{\lambda_i (u)}$. The total
    strength of all non--physical fronts is uniformly bounded by
    \begin{equation*}
      \sum_{\alpha\in\mathcal{NP}}
      \norma{u^\epsilon(t,x_\alpha +)-u^\epsilon(t,x_\alpha -)}
      \le
      \epsilon
      \quad \mbox{for all } t > 0 \,.
    \end{equation*}
  \item Zero--waves are located at the junctions
    $x_\alpha \in \mathcal{I}$. At a zero--wave $x_\alpha$,
    $\alpha \in \mathcal{ZW}$, the traces are related by the coupling
    condition
    $u^+=T\left(\zeta (x_\alpha+), \zeta (x_\alpha-),u^-\right)$ for
    all $t>0$ except at the interaction times.
  \item The initial data satisfies
    $\norma{u^\epsilon(0,\cdot)-u_o}_{\L1(\reali;\reali^n)} \le
    \epsilon$.
  \end{itemize}
\end{definition}

\noindent Next, we prove the existence of $\epsilon$--approximate
solutions.

\begin{theorem}
  \label{thm:ExistenceEpsilonApprox}
  Let $\Omega \subseteq \reali^n$ be open, $f$
  satisfy~\ref{it:f1}--\ref{it:f3} and $\Xi$
  satisfy~\ref{eq:Xi1}--\ref{eq:Xi3}. Fix $\bar u \in A_{i_o}$ and
  $\bar z \in \mathcal{Z}$. Then, there exist $\delta > 0$ such that
  for all piecewise constant $\zeta \in \BV (\reali; \mathcal{Z})$
  with
  \begin{displaymath}
    \zeta (\reali) \subseteq B (\bar z; \delta) \quad \mbox{ and } \quad
    \tv (\zeta) < \delta
  \end{displaymath}
  and for all initial data $u_o $ with
  \begin{equation*}
    u_o (\reali) \subseteq B (\bar u; \delta) \, , \quad
    \tv(u_o) < \delta,
  \end{equation*}
  for every $\epsilon$ sufficiently small there exists an
  $\epsilon$--approximate solution to~\eqref{eq:21} in the sense of
  Definition~\ref{def:WFTApproxSol}. Moreover, the total variation in
  space $\tv(u^\epsilon(t,\cdot))$ and the total variation in time
  $\tv(u^\epsilon(\cdot,x))$ , $x \neq x_\alpha$,
  $\alpha \in \mathcal{Z}\mathcal{W}$ are bounded uniformly for
  $\epsilon$ sufficiently small and for every piecewise constant
  $\zeta$ with $\tv (\zeta) < \delta$.
\end{theorem}

\begin{proof}
  \textit{Description of the wave front tracking algorithm.}  For
  notational convenience, we drop the $\epsilon$. Let $\epsilon$ and
  $\tv (\zeta)$ be sufficiently small, then we construct the
  approximate solution in the following way:
  \begin{itemize}
  \item To obtain piecewise constant approximate solutions, we
    discretize the rarefactions as in~\cite{Bressan2000}. For a fixed
    small parameter $\delta_R$, each rarefaction of size $\sigma$ is
    divided into $m = [[{\sigma}/{\delta_R}]]+1$ wave-fronts, each one
    with size $\sigma/m\leq \delta_R$.
  \item Given initial data $u_o$, we can define a piecewise constant
    approximation $u(0,\cdot)$ satisfying the requirements of
    Definition~\ref{def:WFTApproxSol} and
    \begin{displaymath}
      \tv \left(u(0,\cdot)\right) \leq \tv(u_o) \,.
    \end{displaymath}
    For small $t$, $u(t,x)$ is constructed by solving the generalized
    Riemann problem at every point $x_\alpha$ with
    $\alpha \in \mathcal{ZW}$ and by solving the homogeneous Riemann
    Problem at every remaining discontinuity in $u(0,\cdot)$.
  \item At every interaction point, a new Riemann Problem
    arises. Notice that because of their fixed speed, two
    non--physical fronts cannot interact with each other, neither can
    the zero--waves. Moreover, by a slight modification of the speed
    of some waves (only among shocks, contact discontinuities and
    rarefactions), it is possible to achieve the property that not
    more than two wave-fronts interact at a point.
  \end{itemize}
  After each interaction time, the number of wave-fronts may
  increase. In order to prevent this number to become infinite in
  finite time, a specific treatment has been proposed for waves whose
  strength is below a threshold value $\rho$ by means of a simplified
  Riemann solver~\cite[\S~7.2]{Bressan2000}.

  Suppose that two wave--fronts of strengths $\sigma$, $\sigma'$
  interact at a given point $(t,x)$. If $x \not= x_\alpha$,
  $\alpha \in \mathcal{ZW}$, we use the classical accurate or
  simplified homogeneous Riemann solver as
  in~\cite[\S~7.2]{Bressan2000}. Assume now that $x = x_\alpha$,
  $\alpha \in \mathcal{ZW}$. We briefly recall the different
  situations that can occur, see~\cite{AmadoriGosseGuerra2002} for
  more details.
  \begin{itemize}
  \item If the wave approaching the zero wave is physical and
    $|\sigma \, \sigma'| \ge \rho$ we use the (accurate) generalized
    Riemann solver.
  \item If the wave approaching the zero wave is physical and
    $\modulo{\sigma \,\sigma'} < \rho$, we use a simplified Riemann
    solver. Assume that the wave-front on the right is the
    zero--wave. Let $u_l$, $u_m = H_i(\sigma)(u_l)$,
    $u_r = T(\zeta (x_\alpha+), \zeta (x_\alpha-), u_m)$ be the states
    before the interaction. We define the auxiliary states
    \begin{equation*}
      \tilde{u}_m
      =
      T(\zeta (x_\alpha+), \zeta (x_\alpha-), u_l) \,,
      \qquad
      \tilde{u}_r = H_i(\sigma)(\tilde{u}_m) \,.
    \end{equation*}
    Then, three fronts propagate after the interaction: the zero--wave
    $(u_l, \tilde{u}_m)$, the physical front
    $(\tilde{u}_m, \tilde{u}_r)$ and the non--physical one
    $(\tilde{u}_r,u_r)$.  Due to the commutation defect, we use
    Lemma~\ref{lem:AGG6} to ensure that the introduced error, i.e.~the
    size of the generated non--physical wave, is of second order.
  \item Suppose now that the wave on the left belongs to
    $\mathcal{NP}$. Again we use a simplified solver: let
    $u_l,\ u_m,\ u_r = T(\zeta (x_\alpha+), \zeta (x_\alpha-), u_m)$
    be the states before the interaction and define the new state
    $\tilde{u}_l = T(\zeta (x_\alpha+), \zeta (x_\alpha-),
    u_l)$. After the interaction time, only two fronts propagate: the
    zero--wave $(u_l, \tilde{u}_l)$ and the non--physical wave
    $(\tilde{u}_l,u_r)$. Lemma~\ref{lem:ExistenceT} ensures that the
    error we made is quadratic.
  \end{itemize}

  \textit{Stability of the algorithm.}  We recall how junctions are
  taken care in~\cite{AmadoriGosseGuerra2002}, within the Glimm
  functionals~\cite{Glimm1965}:
  \begin{equation}
    \label{Terhi}
    V(t)
    = \sum_{\alpha \in \mathcal{S \cup R \cup NP \cup ZW}} |\sigma_\alpha|
    \,,\qquad
    Q(t) =  \sum_{\alpha,\beta \in \tilde{\mathcal A}}
    |\sigma_\alpha\sigma_\beta|,
  \end{equation}
  measuring respectively the total wave strengths and the interaction
  potential in $u(t,\cdot)$. Remember that if $\alpha\in\mathcal{ZW}$
  then the strength of the wave located in $x_\alpha$ is given by
  $\sigma_\alpha = \norma{\zeta (x_\alpha+) -\zeta (x_a-)}$.  Notice
  that there exists a constant $C > 1$ (see Lemma~\ref{lem:AGG4}) such
  that
  \begin{displaymath}
    {1\over C}
    \left(
      \tv(u(t,\cdot))
      +
      \tv (\zeta)
    \right)
    \leq
    V(t)
    \leq
    C \left(
      \tv(u(t,\cdot))
      +
      \tv (\zeta)
    \right) \,.
  \end{displaymath}
  Thus, according to the estimates in Lemma~\ref{lem:ExistenceT} and
  Lemma~\ref{lem:AGG7} and to the classical
  ones~\cite[Lemma~7.2]{Bressan2000}, at every time $\tau$ when two
  waves of strengths $\sigma, \sigma'$ interact, we get:
  \begin{eqnarray}
    \label{delta_V}
    V(\tau+) - V(\tau-)
    & \leq
    & C \, \modulo{\sigma\,\sigma'},
    \\
    \label{delta_Q}
    Q(\tau+) - Q(\tau-)
    & \leq
    & (C \, V(\tau-) - 1) \, \modulo{\sigma \, \sigma'}  \,.
  \end{eqnarray}
  Therefore, if $V$ is sufficiently small, \eqref{delta_Q} implies
  \begin{equation}
    \label{delta_Q2}
    Q(\tau+) -Q(\tau-) \le -\frac{1}{2} \, \modulo{\sigma \, \sigma'}\,.
  \end{equation}
  By~\eqref{delta_V} and~\eqref{delta_Q2} we can choose a constant $C$
  large enough and $\delta_* > 0$ so that~\eqref{delta_Q2} holds and
  the quantity
  \begin{equation}
    \label{def_upsilon}
    \Upsilon(t) = V(t) + C \, Q(t)
  \end{equation}
  decreases at every interaction time $\tau$ provided that $V(\tau-)$
  is sufficintly small. Thus, by standard
  arguments~\cite{AmadoriGosseGuerra2002}, choosing initial data $u_o$
  satisfying
  \begin{equation}
    \tv( u_o) + \tv(\zeta) \le \delta \,,
  \end{equation}
  ensures that the $\epsilon$--approximate solution satisfies for any
  $t\ge 0$,
  \begin{equation}
    \label{eq:ineq2}
    \tv(u(t,\cdot)) + \tv (\zeta)
    \le
    \delta_* \,.
  \end{equation}

  The same arguments used in~\cite{AmadoriGosseGuerra2002} allow to
  control the total number of wave fronts, that the maximal strength
  of a rarefaction wavelet is bounded by $\O\,\epsilon$, that the sum
  of the strengths of all $\mathcal{NP}$ waves is also bounded by
  $\O\,\epsilon$ and that $t \to \tv (u (t,x))$, for
  $x \not\in \mathcal{I}$, is bounded uniformly in $\epsilon$ and
  $\zeta$.
\end{proof}

\subsubsection*{Passing to the Limit $\pmb\epsilon \to 0$}

\begin{theorem}
  \label{thm:LimitEpsilonApprox}
  Let $\Omega \subseteq \reali^n$ be open, $f$
  satisfy~\ref{it:f1}--\ref{it:f3} and $\Xi$
  satisfy~\ref{eq:Xi1}--\ref{eq:Xi3}. Fix $\bar u \in A_{i_o}$ and
  $\bar z \in \mathcal{Z}$. Then, there exist $\delta > 0$ such that
  for all piecewise constant $\zeta \in \BV (\reali; \mathcal{Z})$
  with
  \begin{displaymath}
    \zeta (\reali) \subseteq B (\bar z; \delta) \quad \mbox{ and } \quad
    \tv (\zeta) < \delta
  \end{displaymath}
  and for all initial data $u_o $ with
  \begin{equation*}
    u_o (\reali) \subseteq B (\bar u; \delta) \, , \quad
    \tv(u_o) < \delta,
  \end{equation*}
  the Cauchy Problem~\eqref{eq:21} admits a solution $u$ in the sense
  of Definition~\ref{def:sol} enjoying the properties:
  \begin{enumerate}[label={\bf{(\arabic*)}}]
  \item \label{item:1} The maps $t \to \tv (u (t,\cdot))$ and
    $t \to \norma{(u (t,\cdot)}_{\L\infty (\reali; \reali^n)}$ are
    uniformly bounded and the map $x \to u (t,x)$ is left continuous,
    for all $t \geq 0$.
  \item\label{item:2} For all $x \in \reali$, the map $t \to u (t,x)$
    admits a representative $\tilde u_x$ such that $\tv (\tilde u_x)$
    is uniformly bounded.
  \item \label{item:3} For all $t \geq 0$,
    $u (t, \cdot) \in \Lloc1 (\reali; \reali^n)$ and the map
    $t \to u (t,\cdot)$ is $\L1 (\reali; \reali^n)$--Lipschitz
    continuous.
  \item \label{item:4} For all $T>0$ and for all open interval
    $J \subseteq \reali \setminus \mathcal{I}$, the map
    $x \to u (\cdot, x)$ is $\L1 ([0,T]; \reali^n)$--Lipschitz
    continuous, with a Lipschitz constant independent of $J$,
    $\mathcal{I}$ being the set of points of jump of $\zeta$.
  \end{enumerate}
\end{theorem}

\begin{proof}
  For $\epsilon>0$ sufficiently small, fix an $\epsilon$--approximate
  solution $u^\epsilon$. By Theorem~\ref{thm:ExistenceEpsilonApprox},
  $u^\epsilon$
  satisfies~\ref{item:1}--\ref{item:2}--\ref{item:3}--\ref{item:4}. By
  Helly Theorem as extended in~\cite[\S~2.5]{Bressan2000}, there
  exists a map
  $u \colon \mathopen[0, +\infty\mathclose[ \times \reali \to
  \reali^n$ such that, up to a subsequence, $u^\epsilon (t, \cdot)$
  converges to $u (t, \cdot)$ in $\Lloc1 (\reali; \reali^n)$ for all
  $t \in \mathopen[0, +\infty\mathclose[$ and $u$
  satisfys~\ref{item:1}, \ref{item:3}.

  We now prove that $u$ satisfies~\ref{item:4}. By possibly passing to
  subsequences, we may assume that
  $u^\epsilon (\cdot, x) \to u (\cdot,x)$ for a.e.~$x \in \reali$ in
  $\L1 ([0,T]; \reali^n)$. Since $u^\epsilon$ satisfies~\ref{item:4},
  we may pass the Lipschitz continuity of $x \to u^\epsilon (\cdot,x)$
  to the limit $\epsilon \to 0$ for a.e.~$x \in \reali$. The limit $u$
  is left continuous in the space variable $x$ by~\ref{item:1}, hence
  $u$ satisfies~\ref{item:4}.

  We now prove that for all $x \in \reali$,
  $u^\epsilon (\cdot,x) \to u (\cdot,x)$ in $\L1 ([0,T];
  \reali^n)$. To this aim, fix an arbitrary $x \in \reali$ and $y < x$
  such that
  $\mathopen ]y,x\mathclose[ \subset \reali \setminus\mathcal{I}$ and
  $u^\epsilon (\cdot, y) \to u (\cdot, y)$ in $\L1 ([0,T];
  \reali^n)$. Both $x \to u^\epsilon (\cdot,x)$ and
  $x \to u (\cdot,x)$ are Lipschitz continuous, hence
  \begin{eqnarray*}
    \int_0^T \norma{u^\epsilon (t,x) - u (t,x)} \d{t}
    & \leq
    & \O \, \modulo{x-y}
      +
      \int_0^T \norma{u^\epsilon (t,y) - u (t,y)} \d{t} \,;
    \\
    \limsup_{\epsilon\to 0}
    \int_0^T \norma{u^\epsilon (t,x) - u (t,x)} \d{t}
    & \leq
    & \O \, \modulo{x-y} \,.
  \end{eqnarray*}
  Letting now $y \to x$ we obtained the desired convergence.

  Note that $\tv( u^\epsilon (\cdot, x))$ is bounded uniformly, so
  that $u (\cdot,x)$ admits a $\BV$ representative,
  proving~\ref{item:2}.

  Finally, we prove that $u$ solves~\eqref{eq:21}. Choose
  $\phi \in \Cc1 (\mathopen]0,T\mathclose[ \times \reali; \reali)$ and
  $K$ so that
  $\spt \phi \subseteq \mathopen]0,T\mathclose[ \times \mathopen]-K, K
  \mathclose[$. Then,
  \begin{equation*}
    \int_0^T \int_{-K}^K
    \left(u^\epsilon \, \partial_t \phi + f(u^\epsilon) \, \partial_x \phi\right)
    \d x \, \d t
    =
    \int_0^T \sum_{\alpha\in\mathcal J}
    e_{\epsilon,\alpha} (t)\, \phi(t,x_\alpha(t)) \d t,
  \end{equation*}
  where $e_{\epsilon,\alpha}(t)$ measures the error in the
  Rankine--Hugoniot conditions along the discontinuity supported on
  $ x = x_\alpha(t)$, $\alpha \in \mathcal {J}$, i.e.,
  \begin{equation*}
    e_{\epsilon,\alpha} (t)
    =
    \dot x_{\alpha} \left(u^\epsilon(t,x_\alpha(t)+)
      -
      u^\epsilon(t,x_\alpha(t))\right)
    -
    \left(f(u^\epsilon(t,x_\alpha(t)+)) - f(u^\epsilon(t,x_\alpha(t))) \right).
  \end{equation*}
  By Definition~\ref{def:WFTApproxSol} and standard estimates,
  \begin{equation*}
    \sum_{\alpha\in\mathcal{J}\backslash\mathcal{ZW}}|e_{\epsilon,\alpha}(t)|
    \leq
    \O \, \epsilon \,.
  \end{equation*}
  Since the coupling condition~\eqref{eq:DefPsiCoupling} holds along
  the zero--waves $\alpha\in\mathcal{ZW}$, we obtain
  \begin{equation}
    \label{eq:28}
    \norma{
      \int_0^T \!\!\! \int_{-K}^K \!
      \left(u^\epsilon \partial_t \phi+ f(u^\epsilon) \partial_x \phi\right)
      \d x\, \d t
      +
      \int_0^T \!\!\! \sum_{\alpha\in\mathcal{ZW}} \! \phi(t,x_\alpha) \,
      \Xi(\zeta (x_\alpha+), \zeta (x_\alpha), u^\epsilon(t,x_\alpha)) \d t
    }
    \leq
    C \, \epsilon .
  \end{equation}
  As $\epsilon \to 0$ the first integrand above converges to the
  integrand on the left hand side of~\eqref{eq:27}.
  Using~\ref{eq:Xi1} and the convergence
  $u^\epsilon (\cdot, x_\alpha) \to u (\cdot, x_\alpha)$ in
  $\L1 ([0,T];\reali^n)$, we prove the convergence of the second
  integrand in the left hand side of~\eqref{eq:28}, obtaining
  \begin{displaymath}
    \int_0^{+\infty} \!\!\! \int_{\reali} \!
    \left(u \, \partial_t \phi+ f(u) \, \partial_x \phi\right)
    \d x\, \d t
    +
    \int_0^{+\infty} \!\!\! \sum_{\alpha\in\mathcal{ZW}} \! \phi(t,x_\alpha) \,
    \Xi(\zeta (x_\alpha+), \zeta (x_\alpha), u(t,x_\alpha)) \d t
    =0 \,,
  \end{displaymath}
  completing the proof.
\end{proof}

\subsection{Convergence Towards a General $\zeta$}
\label{sub:GeneralConvergence}

\begin{proofof}{Theorem~\ref{thm:main}}
  The proof consists of different steps.

  \paragraph{Step~1: Approximation of $\zeta$.} Let
  $\zeta \in \BV (\reali; \mathcal{Z})$. Call $\mathcal{I}$ the,
  possibly infinite, set of points of jump in $\zeta$. Recall that
  $D\zeta$ is a finite measure. By Lusin
  Theorem~\cite[Theorem~2.24]{Rudin}, for any $h > 0$, there exists a
  $g^h \in \Cc0 (\reali; \reali^p)$ such that $\norma{g^h (x)} \leq 1$
  and
  \begin{equation}
    \label{eq:29}
    \norma{D\zeta}
    \left(
      \left\{x \in \reali \colon g^h (x) \neq v (x)\right\}
    \right) < h \,.
  \end{equation}
  Introduce points $\{ x_1, \ldots, x_{N_h-1}\} \in \reali$ such
  that\footnote{Everywhere, $\sharp A$ stands the (finite) cardinality
    of the set $A$.}:
  \begin{enumerate}[label={\bf{(\roman*)}}]
  \item \label{item:r1} $x_0 = -\infty$, $x_1 < -1/h$, $x_{i-1} < x_i$
    for $i = 2, \ldots, N_h-1$, $x_{N_h-1} > 1/h$ and
    $x_{N_h} = +\infty$.
  \item \label{item:r2}
    $\sum_{x \in \mathcal{I} \setminus \mathcal{I}^h} \left\|\Delta
      \zeta(x)\right\|< h$ for a suitable set of points
    $\mathcal{I}^h$ contained in $\{x_1, x_2, \ldots, x_{N_h -1}\}$.
  \item \label{item:r3} Whenever $x_i \in \mathcal{I}^h$,
    $\tv\left(\zeta, \left[x_{i-1}, x_i \right[\right) < h / (1+\sharp
    \mathcal{I}^h)$.
  \item \label{item:r4}
    $\tv \left(\zeta, \left]x_{i-1}, x_i\right[\right) < h$ for all
    $i = 1, \ldots, N_h$.
  \item \label{item:r6}
    $\norma{g^h \left(x'\right) - g^h \left(x''\right)} < h$ for
    $x',\,x''\in]x_{i-1},x_{i}[$, $i=1,\ldots,N_{h}$.

  \item \label{item:r7} $x_i - x_{i-1} \in \left]0, h\right[$ for all
    $i=2, \ldots, N_h-1$.
  \end{enumerate}
  Such points exist since all these properties are stable if further
  points are inserted.  We approximate $\zeta$ by means of the
  piecewise constant left continuous function
  \begin{equation}
    \label{eq:24}
    \zeta^h (x)
    =
    \zeta\left(-\infty\right)\caratt{]-\infty,x_{1}]}(x)+
    \sum_{i=2}^{N_h-1} \zeta \left(x_{i-1}+\right) \;
    \caratt{]x_{i-1}, x_i]} (x)
    +\zeta\left(x_{N_{h}-1}+\right)\caratt{]x_{N_{h}-1},+\infty [}(x).
  \end{equation}

  By Theorem~\ref{thm:LimitEpsilonApprox}, there exists a solution
  $u^h$ to the Cauchy Problem~\eqref{eq:21} with $\zeta$ substituted
  by $\zeta^h$ as in~\eqref{eq:24}. We prove that as $h\to 0$ the
  solution $u^h$ converges in $\Lloc1$ to a solution to~\eqref{eq:21},
  possibly up to a subsequence.

  \paragraph{Step~2: Select a Convergent Subsequence.} We claim that
  there exists a map $u$ and a convergent subsequence, which we keep
  denoting $u^h$, such that
  \begin{eqnarray}
    \label{eq:30}
    u^h (t,\cdot)
    & \to
    & u (t,\cdot)
      \mbox{ in } \Lloc1 (\reali; \reali^n)
      \mbox{ for all }t;
    \\
    \label{eq:31}
    u^h (\cdot, x)
    & \to
    & u (\cdot, x)
      \mbox{ in } \Lloc1 (\mathopen[0, +\infty\mathclose[; \reali^n)
      \mbox{ for all }x;
    \\
    \label{eq:32}
    t
    & \to
    & u (t, \cdot)
      \mbox{ is Lipschitz continuous in } \L1 (\reali; \reali^n);
    \\
    \label{eq:33}
    \tv\left(u (t, \cdot)\right)
    & \mbox{is}
    & \mbox{bounded uniformly in }t \,;
    \\
    \label{eq:40}
    x \to u (t,x)
    & \mbox{is}
    & \mbox{left continuous for all } t \geq 0 \,.
  \end{eqnarray}
  Indeed, by~\ref{item:1} and~\ref{item:3} in
  Theorem~\ref{thm:LimitEpsilonApprox}, we can apply Helly Theorem as
  presented in~\cite[\S~2.5]{Bressan2000} obtaining the existence of a
  map $u$ satisfying~\eqref{eq:30}, \eqref{eq:32}, \eqref{eq:33}
  and~\eqref{eq:40}.

  We are left with the convergence~\eqref{eq:31}. Introduce a point
  $y <x$ and all the points of jump $\bar x_0, \ldots, \bar x_{M+1}$
  (for a suitable $M\ge 0$) in $\zeta^h$ such that
  \begin{displaymath}
    -\infty \leq \bar x_0  < y \le \bar x_1 < \bar x_2
    < \cdots < \bar x_M < x \leq \bar x_{M+1} \leq +\infty \,.
  \end{displaymath}
  We now estimate
  \begin{eqnarray}
    \nonumber
    \tv\left(\zeta^h; \mathopen[y, x \mathclose[\right)
    & =
    &\sum_{i=1}^M \norma{\Delta\zeta^{h}\left(\bar x_{i}\right)}=
      \sum_{i=1}^M \norma{\zeta (\bar x_i+) - \zeta (\bar x_{i-1}+)}
    \\
    \nonumber
    & \leq
    & \norma{\zeta (\bar x_1-) - \zeta (\bar x_0+)}
      + \norma{\zeta (\bar x_1+) - \zeta (\bar x_1-)}
      + \sum_{i=2}^ M \norma{\zeta (\bar x_i+) - \zeta (\bar x_{i-1}+)}
    \\
    \nonumber
    & \leq
    & \tv\left(\zeta, \mathopen]\bar x_0, \bar x_1\mathclose[\right)
      + \tv \left(\zeta, \mathopen]y, x\mathclose[\right)
    \\
    \label{eq:34}
    & \leq
    & h + \tv \left(\zeta, \mathopen]y, x\mathclose[\right) \,.
  \end{eqnarray}
  where to get to the last line above we used~\ref{item:r4}.

  Fix a positive $T$. By the triangle inequality, \ref{item:4} in
  Theorem~\ref{thm:LimitEpsilonApprox}, inequality~\eqref{eq:34} and
  Lemma~\ref{lem:ExistenceT}, since
  $u^h (t, \bar x_i+) = T\left(\zeta^h (\bar x_i+), \zeta^h(\bar x_i),
    u^h (t, x_i) \right)$,
  \begin{eqnarray}
    \nonumber
    &
    & \int_0^T \norma{u^h (t, x) - u^h (t,y)} \d{t}
    \\
    \nonumber
    & \leq
    & \int_0^T \norma{u^h (t, x) - u^h (t,\bar x_M+)} \d{t}
      + \int_0^T \norma{u^h (t, \bar x_M+) - u^h (t,\bar x_M)} \d{t}
    \\
    \nonumber
    &
    & +
      \sum_{i=1}^{M-1} \left(
      \int_0^T \norma{u^h (t,\bar x_{i+1}) - u^h (t, \bar x_i+)} \d{t}
      +
      \int_0^T \norma{u^h (t,\bar x_i+) - u^h (t, \bar x_i)} \d{t}
      \right)
    \\
    \nonumber
    &
    & +
      \int_0^T \norma{u^h (t, \bar x_1) - u^h (t, y+)} \d{t}
      +
      \int_0^T \norma{u^h (t, y+) - u^h (t, y)} \d{t}
    \\
    \nonumber
    & \leq
    & \O \, \modulo{x - \bar x_M}
      + \O \, \norma{\Delta\zeta^{h} (\bar x_M)}
    \\
    \nonumber
    &
    & +
      \O \, \sum_{i=1}^{M-1}
      \left(\modulo{\bar x_{i+1} - \bar x_i}
      + \norma{\Delta \zeta^{h} (\bar x_i)}
      \right)
      + \O \, \modulo{\bar x_1 - y}
      + \O \, \norma{\Delta\zeta^{h} (y) }
    \\
    \nonumber
    & \leq
    & \O \left( \modulo{x-y} + \tv\left(\zeta^h, \mathopen[y, x\mathclose[\right)\right)
    \\
    \label{eq:35}
    & \leq
    & \O \left(
      \modulo{x-y} + h + \tv\left(\zeta, \mathopen[y, x\mathclose[\right)
      \right) \,.
  \end{eqnarray}
  Since $u^{h}$ converges to $u$ in
  $\Lloc1\left([0,+\infty[\times\mathbb{R},\mathbb{R}^{n}\right)$ too,
  possibly passing to a subsequence, we may assume that for
  a.e.~$x \in \reali$ we have $u^h (\cdot,x) \to u (\cdot,x)$ in
  $\Lloc1 (\mathopen[0, +\infty{\mathclose[; \reali^n})$. Pass to the
  limit $h \to 0$ in~\eqref{eq:35} and obtain that for
  a.e.~$x,y \in \reali$ with $y < x$,
  \begin{equation}
    \label{eq:41}
    \int_0^T \norma{u (t, x) - u (t,y)} \d{t}
    \leq
    \O \, \left(
      \modulo{x-y} + \tv\left(\zeta, \mathopen[y, x\mathclose[\right)
    \right) \,.
  \end{equation}
  By the left continuity of $x \to u (t, x)$ and of the right hand
  side of~\eqref{eq:41} (with respect to both $x$ and $y$), the
  inequality~\eqref{eq:41} holds for \emph{all} $x,y \in \reali$ with
  $y < x$.

  Fix now an arbitrary $x \in \reali$ and choose $y \in \reali$ with
  $y<x$ and such that $u^h (\cdot, y) \to u (\cdot, y)$. By the
  triangle inequality, \eqref{eq:35} and~\eqref{eq:41}, we have
  \begin{displaymath}
    \int_0^T \! \norma{u^h (t,x) - u (t,x)} \d{t}
    \leq
    \O \left(
    \modulo{x-y} + h
    + \tv \left(\zeta, \mathopen[y, x\mathclose[\right)
    \right)
    + \int_0^T \norma{u^h (t,y) - u (t,y)} \d{t} \, .
  \end{displaymath}
  Hence, for almost every $y<x$,
  \begin{displaymath}
    \limsup_{h \to 0} \int_0^T \norma{u^h (t,x) - u (t,x)} \d{t}
    \leq
    \O \left(
      \modulo{x-y}
      +
      \tv \left(\zeta, \mathopen[y, x\mathclose[\right)
    \right)
  \end{displaymath}
  which proves the convergence for every $x \in \reali$, since the
  latter right hand side vanishes as $y \to x^-$.

  \paragraph{Step~3: The Limit is a Solution.}
  Fix
  $\phi \in \Cc1 (\mathopen]0, +\infty\mathclose[ \times \reali;
  \reali)$ such that $\spt \phi \subseteq [0,T] \times [-K, K]$ for
  suitable $T, K >0$. Showing that the left hand side below vanishes
  in the limit $h \to 0$ completes the proof.
  \begin{eqnarray*}
    &
    &
      \left\|
      -\int_0^T \int_{-K}^K
      \left(
      u \, \partial_t \phi + f (u) \, \partial_x \phi
      \right)
      \d{t} \d{x}
      -
      \sum_{x \in \mathcal{I}\,,\, \modulo{x} \leq K}
      \int_0^T \Xi\left(\zeta (x+), \zeta (x), u (t,x)\right)
      \phi (t, x) \d{t}
      \right.
    \\
    &
    &
      \left.
      -
      \int_0^T \int_{-K}^K
      D_{v (x)}^+ \Xi\left(\zeta (x), \zeta (x), u (t,x)\right)
      \phi (t,x)
      \d{\norma{\mu}} (x )\d{t}
      \right\|
    \\
    & \leq
    & \mathcal{E}_1^h + \mathcal{E}_2^h + \mathcal{E}_3^h +
      \mathcal{E}_4^h + \mathcal{E}_5^h + \mathcal{E}_6^h +
      \mathcal{E}_7^h + \mathcal{E}_8^h + \mathcal{E}_9^h +
      \mathcal{E}_{10}^h\,.
  \end{eqnarray*}
  To this aim, consider the terms on the right hand side separately:

  \subparagraph{Term $\mathcal{E}_1^h$:} By the $\Lloc1$ convergence
  proved in Step~2.

  \begin{eqnarray*}
    \mathcal{E}_1^h
    & =
    & \norma{
      -\int_0^T \int_{-K}^K \left(u \, \partial\phi + f (u) \, \partial_x\phi\right) \d{x} \d{t}
      + \int_0^T \int_{-K}^K \left(u^h \, \partial\phi + f (u^h) \, \partial_x\phi\right) \d{x} \d{t}
      }
    \\
    & \to
    & 0  \quad \mbox{ as } h \to 0 \,.
  \end{eqnarray*}

  \subparagraph{Term $\mathcal{E}_2^h$:} Each $u^h$ is a solution,
  hence
  \begin{displaymath}
    -\int_0^T \int_{-K}^K
    \left(
      u^h \, \partial_t \phi + f (u^h) \, \partial_x \phi
    \right)
    \d{t} \d{x}
    =
    \sum_{i \colon \modulo{x_i} \leq K}
    \int_0^T \Xi\left(\zeta^h (x_i+), \zeta^h (x_i), u^h (t,x_i)\right)
    \phi (t, x_i) \d{t}
  \end{displaymath}
  so that
  \begin{eqnarray*}
    \mathcal{E}_2^h
    & =
    & \left\|
      -\int_0^T \int_{-K}^K
      \left(
      u^h \, \partial_t \phi + f (u^h) \, \partial_x \phi
      \right)
      \d{t} \d{x}
      \right.
    \\
    &
    & \qquad
      \left.
      -
      \sum_{i \colon \modulo{x_i} \leq K}
      \int_0^T \Xi\left(\zeta^h (x_i+), \zeta^h (x_i), u^h (t,x_i)\right)
      \phi (t, x_i) \d{t}
      \right\|
    \\
    & =
    & 0 \,.
  \end{eqnarray*}

  \subparagraph{Term $\mathcal{E}_3^h$:} Recall that by~\eqref{eq:24},
  $\zeta^h (x_i) = \zeta (x_{i-1}+)$ and
  $\zeta^h (x_i+) = \zeta (x_i+)$. By the Lipschitz continuity of
  $\Xi$ and \ref{item:r3}
  \begin{eqnarray*}
    \mathcal{E}_3^h
    & =
    & \left\|
      \sum_{i \colon \modulo{x_i}\leq K \,,\, x_i \in \mathcal{I}^h}
      \int_0^T \Xi\left(\zeta^h (x_i+), \zeta^h (x_i), u^h (t, x_i)\right)
      \, \phi (t, x_i) \d{t}
      \right.
    \\
    &
    & \quad - \left.
      \sum_{i \colon \modulo{x_i}\leq K \,,\, x_i \in \mathcal{I}^h}
      \int_0^T \Xi\left(\zeta (x_i+), \zeta (x_i), u^h (t, x_i)\right)
      \, \phi (t, x_i) \d{t}
      \right\|
    \\
    & \leq
    & \O
      \sum_{i \colon \modulo{x_i}\leq K \,,\, x_i \in \mathcal{I}^h}
      \norma{\zeta (x_{i-1}+) - \zeta (x_i)}
    \\
    & \leq
    & \O \; \sharp\mathcal{I}^h \; \dfrac{h}{1+\sharp \mathcal{I}^h}
    \\
    & \leq
    & \O \; h
    \\
    & \to
    & 0
      \quad \mbox{ as } h\to 0 \,.
  \end{eqnarray*}
  \subparagraph{Term $\mathcal{E}_4^h$:} Recall that if
  $x_{i}\not\in\mathcal{I}$, then
  $\zeta\left(x_{i}+\right)=\zeta\left(x_{i}\right)$, which implies
  the equality
  $\Xi\left(\zeta (x_i+), \zeta (x_i), u^h (t, x_i)\right)=0$. Hence,
  by~\ref{eq:Xi1},~\ref{eq:Xi3} and \ref{item:r2} we compute
  \begin{eqnarray*}
    \mathcal{E}_4^h
    & =
    & \left\|
      \sum_{i \colon \modulo{x_i}\leq K \,,\, x_i \in \mathcal{I}^h}
      \int_0^T \Xi\left(\zeta (x_i+), \zeta (x_i), u^h (t, x_i)\right)
      \, \phi (t, x_i) \d{t}
      \right.
    \\
    &
    & \qquad - \left.
      \sum_{x\in\mathcal{I}, \modulo{x}\leq K }
      \int_0^T \Xi\left(\zeta (x+), \zeta (x), u^h (t, x)\right)
      \, \phi (t, x) \d{t}
      \right\|
    \\
    & \leq
    & \left\|
      \sum_{x\in\mathcal{I}\setminus\mathcal{I}^{h}, \modulo{x}\leq K}
      \int_0^T \Xi\left(\zeta (x+), \zeta (x), u^h (t, x)\right)
      \, \phi (t, x) \d{t}
      \right\|
    \\
    & \leq
    & \O \; \sum_{x\in\mathcal{I}\setminus\mathcal{I}^{h}, \modulo{x}\leq K}\left\|\Delta\zeta\left(x\right)\right\|
    \\
    & \leq
    & \O \; h
    \\
    & \to
    & 0
      \quad \mbox{ as } h\to 0 \,.
  \end{eqnarray*}
  \subparagraph{Term $\mathcal{E}_5^h$:} Using
  Lemma~\ref{lem:EstimateXi2}
  \begin{eqnarray*}
    \mathcal{E}_5^h
    & =
    & \left\|
      \sum_{x\in\mathcal{I}, \modulo{x}\leq K }
      \int_0^T \Xi\left(\zeta (x+), \zeta (x), u^h (t, x)\right)
      \, \phi (t, x) \d{t}
      \right.
    \\
    &
    & \qquad - \left.
      \sum_{x\in\mathcal{I}, \modulo{x}\leq K }
      \int_0^T \Xi\left(\zeta (x+), \zeta (x), u(t, x)\right)
      \, \phi (t, x) \d{t}
      \right\|
    \\
    & \leq
    & \O \sum_{x\in\mathcal{I}}
      \left(
      \left\|\Delta\zeta(x)\right\| \;
      \int_{0}^{T}\left\|u^h(t, x)
      -u(t,x)\right\|\d{t}
      \right)
    \\
    & \to
    & 0
      \quad \mbox{ as } h\to 0 \,.
  \end{eqnarray*}
  The last limit is due to~\eqref{eq:31} and the convergence of the
  series $\sum_{x\in\mathcal{I}} \left\|\Delta \zeta(x)\right\|$. This
  concludes the convergence to the discrete part of the measure.

  \subparagraph{Term $\mathcal{E}_6^h$.} Recall that by~\eqref{eq:24}
  we have $\zeta^h (x_i) = \zeta (x_{i-1}+)$ and
  $\zeta^h (x_i+) = \zeta (x_i+)$. We use below also~\ref{item:r2}:
  \begin{eqnarray*}
    \mathcal{E}_6^h
    & =
    & \left\|
      \sum_{i \colon \modulo{x_i}\leq K \,,\, x_i \not\in \mathcal{I}^h}
      \int_0^T \Xi\left(\zeta^h (x_i+), \zeta^h (x_i), u^h (t, x_i)\right)
      \, \phi (t, x_i) \d{t}
      \right.
    \\
    &
    & \quad - \left.
      \sum_{i \colon \modulo{x_i}\leq K \,,\, x_i \not\in \mathcal{I}^h}
      \int_0^T \Xi\left(\zeta (x_{i-1}+) + \mu (\mathopen] x_{i-1}, x_i \mathopen [), \zeta (x_{i-1}+), u^h (t, x_i)\right)
      \, \phi (t, x_i) \d{t}
      \right\|
    \\
    & \leq
    & \O \sum_{i \colon \modulo{x_i}\leq K \,,\, x_i \not\in \mathcal{I}^h}
      \norma{
      \zeta (x_i+) - \zeta (x_{i-1}+) - \mu (\mathopen]x_{i-1}, x_i \mathclose[)
      }
    \\
    & =
    & \O \sum_{i \colon \modulo{x_i}\leq K \,,\, x_i \not\in \mathcal{I}^h}
      \norma{D \zeta (\mathopen]x_{i-1}, x_i\mathclose])
      - \mu (\mathopen]x_{i-1}, x_i \mathclose[)}
    \\
    & \leq
    & \O \sum_{x \in \mathcal{I} \setminus \mathcal{I}^h} \norma{\Delta\zeta (x)}
    \\
    & =
    &  \O \, h
    \\
    & \to
    & 0  \quad \mbox{ as } h \to 0 \,.
  \end{eqnarray*}

  \subparagraph{Term $\mathcal{E}^h_7$.} Using~\ref{item:r3},

  \begin{eqnarray*}
    \mathcal{E}_7^h
    & =
    & \left\|
      \sum_{i \colon \modulo{x_i}\leq K \,,\, x_i \not\in \mathcal{I}^h}
      \int_0^T \Xi\left(\zeta (x_{i-1}+) + \mu (\mathopen] x_{i-1}, x_i \mathopen [), \zeta (x_{i-1}+), u^h (t,
      x_i)\right)
      \, \phi (t, x_i) \d{t}
      \right.
    \\
    &
    & \left.
      -
      \sum_{i \colon \modulo{x_i}\leq K}
      \int_0^T \Xi\left(\zeta (x_{i-1}+) + \mu (\mathopen] x_{i-1}, x_i \mathopen [), \zeta (x_{i-1}+), u^h (t,
      x_i)\right)
      \, \phi (t, x_i) \d{t}
      \right\|
    \\
    & \leq
    & \O \sum_{i \colon \modulo{x_i}\leq K\,,\, x_i \in \mathcal{I}^h}
      \tv\left(\zeta; \mathopen ] x_{i-1}, x_i \mathclose[\right)
    \\
    & \leq
    & \O \, \sharp\mathcal{I}^h \, \dfrac{h}{1+\sharp\mathcal{I}^h}
    \\
    & \leq
    & \O \, h
    \\
    & \to
    & 0  \quad \mbox{ as } h \to 0 \,.
  \end{eqnarray*}

  \subparagraph{Term $\mathcal{E}_8^h$.} Introduce now
  $\delta_i = \norma{\mu} (\mathopen]x_{i-1}, x_i\mathclose[)$,
  $\mathcal{J} = \left\{i \in \{1, \ldots, N_h\} \colon \delta_i \neq
    0\right\}$ and for $i \in \mathcal{J}$, let
  $v_i = \mu (\mathopen]x_{i-1}, x_i \mathclose[) / \delta_i$. Below,
  we use~\ref{eq:Xi4} with $\delta_i$ for $t$ and $v_i$ for $v$,
  and~\ref{item:r4}:
  \begin{eqnarray*}
    \mathcal{E}_8^h
    & \leq
    & \left\|
      \sum_{i \colon \modulo{x_i}\leq K}
      \int_0^T
      \Xi\left(\zeta (x_{i-1}+) + \mu (\mathopen] x_{i-1}, x_i \mathopen [), \zeta (x_{i-1}+), u^h (t,x_i)\right)
      \, \phi (t, x_i) \d{t}
      \right.
    \\
    &
    & \qquad \left.
      -
      \sum_{i \colon \modulo{x_i}\leq K\,,\,i \in \mathcal{J}}
      \int_0^T \delta_i \, D^+_{v_i}
      \Xi \left(
      \zeta (x_{i-1}+), \zeta (x_{i-1}+), u^h (t, x_i)\right) \, \phi (t, x_i)
      \d{t}
      \right\|
    \\
    & \leq
    & \O \, \sum_{i \colon \modulo{x_i}\leq K\,,\,i \in \mathcal{J}}
      \int_0^T
      \left\|
      \Xi\left(
      \zeta (x_{i-1}+) + \mu (\mathopen] x_{i-1}, x_i \mathopen [), \zeta (x_{i-1}+), u^h (t,x_i)
      \right)
      \right.
    \\
    &
    & \qquad\qquad\qquad\qquad\qquad
      \left.
      -
      \delta_i \, D^+_{v_i}
      \Xi \left(
      \zeta (x_{i-1}+), \zeta (x_{i-1}+), u^h (t, x_i)\right)
      \right\|
      \d{t}
    \\
    & \leq
    & \O \, \sum_{i \colon \modulo{x_i}\leq K\,,\,i \in \mathcal{J}}
      \sigma (\delta_i) \, \delta_i
    \\
    & \leq
    & \O \, \sigma (h) \, \tv (\zeta)
    \\
    & \to
    & 0 \quad \mbox{ as } h \to 0 \,.
  \end{eqnarray*}

  \subparagraph{Term $\mathcal{E}^h_9$.} Use~\ref{eq:Xi4} and recall
  that by~\eqref{eq:20},
  $v_i = (1/\delta_i) \int_{\mathopen]x_{i-1}, x_i\mathclose[} v (y)
  \d{\norma{\mu}} \! (y)$, while clearly
  $v (x) = (1/\delta_i) \int_{\mathopen]x_{i-1}, x_i\mathclose[} v (x)
  \d{\norma{\mu}} \! (y)$. We also use $g^h$, that is defined
  in~\textbf{Step~1} and satisfies~\eqref{eq:29}.
  \begin{eqnarray}
    \nonumber
    \mathcal{E}^h_9
    & =
    & \left\|
      \sum_{i \colon \modulo{x_i}\leq K\,,\,i \in \mathcal{J}}
      \int_0^T \delta_i \, D^+_{v_i}
      \Xi \left(
      \zeta (x_{i-1}+), \zeta (x_{i-1}+), u^h (t, x_i)\right) \, \phi (t, x_i)
      \d{t}
      \right.
    \\
    \nonumber
    &
    & \left.
      -
      \sum_{i \colon \modulo{x_i}\leq K\,,\,i \in \mathcal{J}}
      \int_0^T \int_{\mathopen]x_{i-1},x_i\mathclose[}
      D^+_{v (x)} \Xi
      \left(\zeta (x_{i-1}+), \zeta (x_{i-1}+), u^h (t, x_i)\right)
      \, \phi (t,x_i)
      \d{\norma{\mu}} \! (x) \d{t}
      \right\|
    \\
    \nonumber
    & \leq
    & \O \sum_{i \colon \modulo{x_i}\leq K\,,\,i \in \mathcal{J}}
      \int_{\mathopen]x_{i-1},x_i\mathclose[}
      \norma{v (x) - v_i} \d{\norma{\mu}} \! (x)
    \\
    \nonumber
    & \leq
    & \O \sum_{i \colon \modulo{x_i}\leq K\,,\,i \in \mathcal{J}}
      \dfrac{1}{\delta_i}
      \int_{\mathopen]x_{i-1},x_i\mathclose[^2}
      \norma{v (x) - v (y)} \d{(\norma{\mu}\otimes\norma{\mu})} \! (x,y)
    \\
    \label{eq:43}
    & \leq
    & \O \!\!\!\! \sum_{i \colon \modulo{x_i}\leq K\,,\,i \in \mathcal{J}}
      \dfrac{1}{\delta_i}
      \int_{\mathopen]x_{i-1},x_i\mathclose[^2} \!\!
      \left[
      \norma{v (x) - g^h (x)}
      {+} \norma{g^h (y) - v (y)}
      \right] \!
      \d{(\norma{\mu}\otimes\norma{\mu})} \! (x,y)
    \\
    \label{eq:44}
    &
    & + \; \O \!\!\!\! \sum_{i \colon \modulo{x_i}\leq K\,,\,i \in \mathcal{J}}
      \dfrac{1}{\delta_i}
      \int_{\mathopen]x_{i-1},x_i\mathclose[^2} \!
      \norma{g^h (x) - g^h (y)}
      \d{(\norma{\mu}\otimes\norma{\mu})} \! (x,y) \,.
  \end{eqnarray}
  The two terms in the integral in~\eqref{eq:43} are estimated in the
  same way, using~\eqref{eq:29}, as
  \begin{eqnarray*}
    &
    &    \sum_{i \colon \modulo{x_i}\leq K\,,\,i \in \mathcal{J}}
      \int_{\mathopen]x_{i-1},x_i\mathclose[^2}
      \dfrac{1}{\delta_i}
      \norma{v (x) - g^h (x)}
      \d{(\norma{\mu}\otimes\norma{\mu})} \! (x,y)
    \\
    & \leq
    & \sum_{i \colon \modulo{x_i}\leq K\,,\,i \in \mathcal{J}}
      \int_{\mathopen]x_{i-1},x_i\mathclose[}
      \norma{v (x) - g^h (x)}
      \d{\norma{\mu}} \! (x)
    \\
    & \leq
    & \int_{\reali}
      \norma{v (x) - g^h (x)}
      \d{\norma{\mu}} \! (x)
    \\
    & \leq
    & \int_{\left\{x \in\reali\colon v (x) \neq g^h (x)\right\}}
      \left(\norma{v (x)} + \norma{g^h (x)}\right)
      \d{\norma{\mu}} \! (x)
    \\
    &\leq
    & 2 \, h
    \\
    & \to
    & 0 \quad \mbox{ as } h \to 0\,.
  \end{eqnarray*}
  We now estimate the term~\eqref{eq:44} by means of~\ref{item:r6}:
  \begin{eqnarray*}
    &
    & \sum_{i \colon \modulo{x_i}\leq K\,,\,i \in \mathcal{J}}
      \dfrac{1}{\delta_i}
      \int_{\mathopen]x_{i-1},x_i\mathclose[^2}
      \norma{g^h (x) - g^h (y)}
      \d{(\norma{\mu}\otimes\norma{\mu})} \! (x,y)
    \\
    & \leq
    & h
      \sum_{i \colon \modulo{x_i}\leq K\,,\,i \in \mathcal{J}}
      \dfrac{1}{\delta_i}
      \int_{\mathopen]x_{i-1},x_i\mathclose[^2}
      \d{(\norma{\mu}\otimes\norma{\mu})} \! (x,y)
    \\
    & \leq
    & h
      \sum_{i \colon \modulo{x_i}\leq K\,,\,i \in \mathcal{J}}
      \delta_i
    \\
    & \leq
    & h \, \tv (\zeta)
    \\
    & \to
    & 0 \quad \mbox{ as } h \to 0\,.
  \end{eqnarray*}

  \subparagraph{Term $\mathcal{E}^h_{10}$.} Using~\ref{eq:Xi4}
  \begin{eqnarray*}
    \mathcal{E}^h_{10}
    & =
    & \left\|
      \sum_{i \colon \modulo{x_i}\leq K\,,\,i \in \mathcal{J}}
      \int_0^T \int_{\mathopen]x_{i-1},x_i\mathclose[} \!
      D^+_{v (x)} \Xi
      \left(\zeta (x_{i-1}+), \zeta (x_{i-1}+), u^h (t, x_i)\right)
      \, \phi (t,x_i)
      \d{\norma{\mu}} \! (x) \d{t}
      \right.
    \\
    &
    & \left.
      \qquad -
      \int_0^T \int_{-K}^K
      D_{v (x)}^+ \Xi\left(\zeta (x), \zeta (x), u (t,x)\right)
      \phi (t,x)
      \d{\norma{\mu}} (x )\d{t}
      \right\|
    \\
    & =
    & \left\|
      \sum_{i \colon \modulo{x_i}\leq K\,,\,i \in \mathcal{J}}
      \int_0^T \int_{\mathopen]x_{i-1},x_i\mathclose[} \!
      D^+_{v (x)} \Xi
      \left(\zeta (x_{i-1}+), \zeta (x_{i-1}+), u^h (t, x_i)\right)
      \! \phi (t,x_i)
      \d{\norma{\mu}} \! (x) \d{t}
      \right.
    \\
    &
    & \qquad
      \left.
      -
      \sum_{i \colon \modulo{x_i}\leq K\,,\,i \in \mathcal{J}}
      \int_0^T \int_{\mathopen]x_{i-1},x_i\mathclose[}
      D_{v (x)}^+ \Xi\left(\zeta (x), \zeta (x), u (t,x)\right)
      \phi (t,x)
      \d{\norma{\mu}} (x )\d{t}
      \right\|
    \\
    & \leq
    & \sum_{i \colon \modulo{x_i}\leq K\,,\,i \in \mathcal{J}}
      \int_0^T \int_{\mathopen]x_{i-1},x_i\mathclose[}
      \left\|
      D^+_{v (x)} \Xi
      \left(\zeta (x_{i-1}+), \zeta (x_{i-1}+), u^h (t, x_i)\right)
      \, \phi (t,x_i)
      \right.
    \\
    &
    & \qquad
      \left.
      \qquad\qquad
      -
      D_{v (x)}^+ \Xi\left(\zeta (x), \zeta (x), u (t,x)\right)
      \phi (t,x)
      \right\|
      \d{\norma{\mu}} (x )\d{t}
    \\
    & \leq
    & \O \sum_{i \colon \modulo{x_i}\leq K\,,\,i \in \mathcal{J}}
      \int_0^T \int_{\mathopen]x_{i-1},x_i\mathclose[}
      \left(
      \norma{\zeta (x_{i-1}+) - \zeta (x)}
      +
      \norma{u^h (t,x_i) - u^h (t,x)}
      \right.
    \\
    &
    & \qquad\qquad\qquad\qquad\qquad
      \left.
      +
      \norma{u^h (t,x) - u (t,x)}
      +
      \modulo{x_i-x}
      \right)
      \d{\norma{\mu}} (x )\d{t} \,.
  \end{eqnarray*}
  Observe that
  \begin{displaymath}
    \begin{array}{rcl@{\qquad}l}
      \norma{\zeta (x_{i-1}+) - \zeta (x)}
      & \leq
      & h
      & \mbox{by~\ref{item:r4}}
      \\
      \int_0^T \norma{u^h (t,x_i) - u^h (t,x)} \d{t}
      & \leq
      & \O \, h
      & \mbox{by~\eqref{eq:35}}
      \\
      \modulo{x_i-x}
      & \leq
      & h
      & \mbox{by~\ref{item:r7}}
    \end{array}
  \end{displaymath}
  while by~\eqref{eq:31}, Fubini Theorem and the Dominated Convergence
  Theorem,
  \begin{displaymath}
    \int_{\reali} \int_0^T
    \norma{u^h (t,x) - u (t,x)} \,
    \d{t} \,
    \d{\norma{\mu}} (x )
    \to 0 \quad \mbox{ as } h \to 0\,.
  \end{displaymath}
  The proof is completed.
\end{proofof}

\subsection{Proof Relative to Section~\ref{sec:Appl}}
\label{sec:Applications}

\begin{proofof}{Theorem~\ref{thm:ConvergenceIsenCurvedPipe}}
  It is immediate to check that~\ref{it:f1}--\ref{it:f3} hold, thanks
  to~\ref{it:p}. Define $\Xi$ as in~\eqref{eq:47}. Then,
  conditions~\ref{eq:Xi1} and~\ref{eq:Xi2} follow from the assumed
  $\C2$ regularity of $K$ in all its variables. Condition~\ref{eq:Xi3}
  follows from~\eqref{eq:47} and $K\left(0,(\rho,q)\right) \equiv
  0$. Concerning~\ref{eq:Xi4}, we have
  \begin{displaymath}
    D_v^+\Xi (z,z,v)
    =
    \left[
      \begin{array}{c}
        0
        \\
        \partial_1 K\left(0, (\rho,q)\right) \, \norma{v}
      \end{array}
    \right]
  \end{displaymath}
  indeed, for $v$ such that $\norma{v} \leq 1$, we can estimate
  \begin{eqnarray*}
    &
    & \norma{
      K\left(t \, \norma{v}, (\rho,q)\right)
      -
      \norma{v} \, \partial_1 K \left(0, (\rho,q)\right) \, t
      }
    \\
    & =
    &
      \norma{
      \int_0^1 \left(
      \partial_1 K\left(s\, t \, \norma{v}, (\rho,q)\right)
      -
      \partial_1 K\left(0, (\rho,q)\right)
      \right)
      t \, \norma{v} \d{s}
      }
    \\
    & \leq
    & \norma{K}_{\C2 ([0,r]\times\Omega; \reali)} \, t^2
  \end{eqnarray*}
  proving~\ref{eq:Xi4} with
  $\sigma (t) = \norma{K}_{\C2 ([0,r]\times\Omega; \reali)} \, t$.

  Theorem~\ref{thm:main} can then be applied, exhibiting the existence
  of a solution in the sense of Definition~\ref{def:sol}.

  To obtain the formulation~\eqref{eq:49} from~\eqref{eq:48}, only the
  two terms in the right hand side of the second equations need to be
  considered. The first one is immediate: it only requires the
  substitution~\eqref{eq:47}. Concerning the second one, recall that
  by~\eqref{eq:20}, $\d\mu (x) = \Gamma'' (x) \d{x}$, so that
  $\d{\norma{\mu}} (x) = \norma{\Gamma'' (x)} \d{x}$, and that
  $v (x) = \frac{\Gamma'' (x)}{\norma{\Gamma'' (x)}}$ for a.e.~$x$
  with respect to the measure $\norma{\mu}$.

  Hence, since $\norma{v (x)}=1$ for a.e.~$x$ with respect to the
  measure $\norma{\mu}$,
  \begin{displaymath}
    D^+_{v (x)} \Xi\left(\Gamma' (x), \Gamma' (x), (\rho,q)\right)
    \d{\norma{\mu}} (x)
    =
    \partial_1 K \left(0,(\rho,q) (x)\right) \,  \norma{\Gamma'' (x)}
    \d{x}
  \end{displaymath}
  completing the proof.
\end{proofof}

\begin{proofof}{Theorem~\ref{thm:VarSec}}
  Condition~\ref{it:p} ensures that~\ref{it:f1}--\ref{it:f3} hold. The
  choice~\eqref{eq:54} and the assumptions on $\Xi_2$ imply
  that~\ref{eq:Xi1}--\ref{eq:Xi4} hold. Since the distributional
  derivative of $a$ has neither Cantor part nor atomic part, due
  to~\eqref{eq:56} problem~\eqref{eq:53} reduces to~\eqref{eq:51}.
\end{proofof}

\noindent\textit{Acknowledgment.} The first and second authors
were partly supported by the GNAMPA~2020 project \emph{"From
  Wellposedness to Game Theory in Conservation Laws"}. The work of the
third author has been funded by the Deutsche Forschungsgemeinschaft
(DFG, German Research Foundation) Projektnummer 320021702/GRK2326
Energy, Entropy, and Dissipative Dynamics (EDDy).

{ \small

  \bibliography{curve_03}

  \bibliographystyle{abbrv}

}

%
%
%
%
%
%
%
%
%

\end{document}